%% file: ArXivVersion.tex
\documentclass[10pt,a4paper,reqno]{amsart}

\usepackage{amsmath,amssymb,amsthm,amsfonts}
\usepackage{amsbsy}
\usepackage[utf8]{inputenc}
\usepackage{indentfirst}
\usepackage{geometry}
\usepackage{enumitem}
\usepackage{a4wide}
\usepackage[english]{babel}
\usepackage{longtable}
\usepackage{url}
\usepackage[all]{xy}
\usepackage{subfig}
\usepackage{aliascnt}

\usepackage{verbatim}

\usepackage{tikz}
\usetikzlibrary{matrix}
\usetikzlibrary{arrows}
\usetikzlibrary{fit}
\usetikzlibrary{shapes}

\usepackage{booktabs}

\definecolor{LinkColor}{rgb}{0,0,0} %black
\usepackage[colorlinks=true,linkcolor=LinkColor,citecolor=LinkColor,urlcolor= LinkColor, naturalnames, hyperindex, pdfstartview=FitH, bookmarksnumbered, plainpages]{hyperref} 

\usepackage[nameinlink, capitalise]{cleveref}

\newtheorem{maintheorem}{Theorem}
\crefname{maintheorem}{Theorem}{Theorems}

\newtheorem{maincorollary}[maintheorem]{Corollary}
\crefname{maincorollary}{Corollary}{Corollaries}

\newtheorem{theorem}{Theorem}[section]
\crefname{theorem}{Theorem}{Theorems}

\newaliascnt{corollary}{theorem}

\aliascntresetthe{corollary}
\crefname{corollary}{Corollary}{Corollaries}

\newaliascnt{lemma}{theorem}
\newtheorem{lemma}[lemma]{Lemma}
\aliascntresetthe{lemma}
\crefname{lemma}{Lemma}{Lemmata}

\newaliascnt{proposition}{theorem}
\newtheorem{proposition}[proposition]{Proposition}
\aliascntresetthe{proposition}
\crefname{proposition}{Proposition}{Propositions}

\newtheorem*{zc*}{Zassenhaus Conjecture (ZC)}
\newtheorem*{pq*}{Prime Graph Question (PQ)}
\newtheorem*{theorem*}{Theorem}

\theoremstyle{definition}

\newaliascnt{definition}{theorem}

\aliascntresetthe{definition}
\crefname{definition}{Definition}{Definitions}

\theoremstyle{remark}

\newaliascnt{example}{theorem}
\newtheorem{example}[example]{Example}
\aliascntresetthe{example}
\crefname{example}{Example}{Examples}

\newaliascnt{remark}{theorem}
\newtheorem{remark}[remark]{Remark}
\aliascntresetthe{remark}
\crefname{remark}{Remark}{Remarks}

\newaliascnt{notation}{theorem}

\aliascntresetthe{notation}
\crefname{notation}{Notation}{Notations}

\newcommand{\Sz}{\operatorname{Sz}}
\newcommand{\PSL}{\operatorname{PSL}}
\newcommand{\PGL}{\operatorname{PGL}}
\newcommand{\PSU}{\operatorname{PSU}}
\newcommand{\PSp}{\operatorname{PSp}}

\newcommand{\Aut}{\operatorname{Aut}}
\newcommand{\Inn}{\operatorname{Inn}}
\newcommand{\Gal}{\operatorname{Gal}}
\newcommand{\Tr}{\operatorname{Tr}}
\newcommand{\FF}{\mathbb{F}}

\newcommand{\V}{\textup{V}}
\newcommand{\ZZ}{\mathbb{Z}}
\newcommand{\QQ}{\mathbb{Q}}
\newcommand{\cc}[1]{\ensuremath{{\texttt{#1}}}}
\newlength{\heightofhw}
\newcommand{\eigbox}[2]{\ensuremath{\textstyle{#1} \times \boxed{\rule{0cm}{\heightofhw} #2}}}

\title[Prime Graph Question for Integral Group Rings of 4-primary groups I]{On the Prime Graph Question for Integral Group Rings\\ of 4-primary groups I}
\author{Andreas B\"achle}
\address{Vakgroep Wiskunde, Vrije Universiteit Brussel, Pleinlaan 2, 1050 Brussels, Belgium}
\email{\href{mailto:abachle@vub.ac.be}{abachle@vub.ac.be}}
\author{Leo Margolis}
\address{Departamento de matem\'aticas, Facultad de matem\'aticas, Universidad de Murcia, 30100 Murcia, Spain}
\email{\href{mailto:leo.margolis@um.es}{leo.margolis@um.es}}
\thanks{The first author is a postdoctoral researcher of the FWO (Research Foundation Flanders). The second author is supported by a Marie Curie grant from EU project 705112-ZC}

\begin{document}

\begin{abstract} We study the Prime Graph Question for integral group rings. This question can be reduced to almost simple groups by a result of Kimmerle and Konovalov. We prove that the Prime Graph Question has an affirmative answer for all almost simple groups having a socle isomorphic to $\PSL(2, p^f)$ for $f \leq 2$, establishing the Prime Graph Question for all groups where the only non-abelian composition factors are of the aforementioned form.  Using this, we determine exactly how far the so-called HeLP method can take us for (almost simple) groups having an order divisible by at most $4$ different primes.
\end{abstract}

\maketitle

\section{Introduction}

Let $G$ be a finite group. The integral group ring $\ZZ G$ of $G$ comes with a natural augmentation map 
\[ \varepsilon \colon \ZZ G \to \ZZ \colon \sum_{g \in G} z_g g \mapsto \sum_{g \in G} z_g. \] 
This map is a ring homomorphism, hence every unit of $\ZZ G$ is mapped to $1$ or $-1$, so up to a sign every unit of $\ZZ G$ lies in $\V(\ZZ G)$ -- the units of augmentation $1$, also called normalized units.

The connections between the properties of $G$ and $\V(\ZZ G)$ inspired a lot of interesting research, culminating in Weiss' results on $p$-adic group rings of nilpotent groups \cite{Weiss88,Weiss91}, Hertweck's counterexample to the isomorphism problem \cite{HertweckAnnals} and Jespers' and del Río's comprehensive books \cite{EricAngel1,EricAngel2}, among others. Concerning the finite subgroups of $\V(\ZZ G)$ the major open question today is the Zassenhaus Conjecture \cite{Zassenhaus}, also known as the first Zassenhaus Conjecture:

\begin{zc*} For any $u \in \V(\ZZ G)$ of finite order, there exist an element $g \in G$ and a unit $x$ in the rational group algebra $\QQ G$ such that $x^{-1}ux = g$. \end{zc*}

In case such $g$ and $x$ exist for a torsion unit $u$, the elements $g$ and $u$ are said to be rationally conjugate. This conjecture has been proven, e.g.\ for all nilpotent groups \cite{Weiss91}, groups having a normal Sylow subgroup with abelian complement \cite{HertweckColloq} and cyclic-by-abelian groups \cite{CyclicByAbelian}. The conjecture is known for the simple groups $\PSL(2, p)$ for primes $p \leq 23$, see \cite{Gitter} for references. But in general this conjecture remains widely open. As a first step towards the Zassenhaus Conjecture W.~Kimmerle proposed the following question \cite{KimmiPQ}:

\begin{pq*} If $\V(\ZZ G)$ contains a unit of order $pq$, does $G$ have an element of order $pq$, where $p$ and $q$ are different primes? \end{pq*}

This is the same as to ask whether $G$ and $\V(\ZZ G)$ have the same Prime Graph. The \emph{Prime Graph} of a group $X$ is the graph $\Gamma(X)$ whose vertices are the primes appearing as orders of elements of $X$; two different vertices $p$ and $q$ are joined by an edge if there is an element of order $pq$ in $X$. Kimmerle was probably motivated to raise this question as he proved it true for all solvable groups \cite{KimmiPQ} -- a class for which a solution for the Zassenhaus Conjecture seems to be out of reach.

The Prime Graph Question became even more approachable by a reduction result \cite[Proposition 4.1]{KonovalovKimmiStAndrews} being equivalent to the following theorem.

\begin{theorem}[Kimmerle, Konovalov]\label{reduction} If the Prime Graph Question has an affirmative answer for all almost simple images of $G$, then it has an affirmative answer for $G$ itself. \end{theorem}

A group $G$ is called \emph{almost simple} if it is ``sandwiched'' between a non-abelian simple group $S$ and its automorphism group, i.e.\ $S \simeq \Inn(S) \leq G \leq \Aut(S)$. In this case $S$ is the socle of $G$.

The Prime Graph Question has been confirmed for all Frobenius groups \cite{KimmiPQ} and for all $\PSL(2, p)$, $p$ a prime \cite{HertweckBrauer}. Furthermore, it is known to hold for almost simple groups with socle $S$ isomorphic to one of the 13 smaller sporadic simple groups (by work of Bovdi, Konovalov and their collaborators, cf.\ \cite{KonovalovKimmiStAndrews} and the references therein) or to an alternating group of degree at most $17$ \cite{Symmetric}. \cref{TheoremPSL2p2} gives together with \cref{reduction} for the first time a positive answer to the Prime Graph Question for finite groups with non-abelian composition factors from infinite series of simple groups.
\begin{maintheorem}\label{TheoremPSL2p2}
The Prime Graph Question has an affirmative answer for all almost simple groups with socle $\PSL(2,p^f)$, if $f \leq 2$.
\end{maintheorem}
The theorem follows from \cref{PGL2,PSL_2_p2}. This, and most of the previously mentioned, results are achieved by extensively using the so-called HeLP method, which is discussed in \cref{HeLP_section}. For a finite group $G$ denote by $\pi(G)$ the set of prime divisors of the order of $G$. A group is called \emph{$k$-primary}, if $|\pi(G)| = k$. After Kimmerle and Konovalov had obtained their reduction result, they investigated in the same article (almost simple) $3$-primary groups. This seemed promising as there are only eight simple groups of this kind by \cite[§ 7]{Feit} and the Classification on Finite Simple Groups. The HeLP method turned out to be not entirely sufficient and the result was completed by the authors using a new method, the so-called lattice method \cite{Gitter}.

 After the Prime Graph Question for $3$-primary groups was settled, it was natural to consider the question for $4$-primary groups -- a much bigger class of groups. Lately, the prime graphs of the almost simple $4$-primary groups were investigated in \cite{KonKh11}. The simple 4-primary groups were first classified in \cite{Shi} and independently in \cite{HuppertLempken}, using the Classification of the Finite Simple Groups. More arithmetic restrictions were obtained in \cite[Theorem 2]{4primaryChina}. There are three, potentially infinite, series and $37$ specific simple $4$-primary groups. These simple groups give rise to $123$ specific and $5$ series of almost simple $4$-primary groups. We sum up the results relevant to us in the following proposition:

\begin{proposition}[Shi; Huppert, Lempken; Bugeaud, Cao, Mignotte] \label{4primaryProposition}
The almost simple 4-primary groups $G$ are the ones listed in \cref{HeLPTable} on page \pageref{HeLPTable}. If $G = \PSL(2,2^f)$, then either $f = 4$ or $2^f-1$ is a Mersenne prime; in this case $\frac{2^f+1}{3}$ is also a prime. If $G = \PSL(2,3^f)$ or $G = \PGL(2,3^f)$, then either $f \in \{4,5\}$ or $\frac{3^f-1}{2}$ and $\frac{3^f+1}{4}$ are both prime.
\end{proposition}

It is an open number theoretical problem, whether there are actually infinitely many 4-primary groups as noted in the Kourovka Notebook \cite[Problem 13.65]{Kourovka}.

This study of the Prime Graph Question for $4$-primary groups is divided into two parts. In this first part we apply the HeLP method to the almost simple $4$-primary groups and determine exactly for which of them the HeLP method suffices to answer the Prime Graph Question. For a bunch of groups, this is straight forward and can be done by a \textsf{GAP} package developed by the authors for such a purpose \cite{HeLPPackage, HeLPPaper}. In these instances, the data are collected in a compact form in \cref{results}. However, there are cases left requiring a more detailed study for different reasons. Here, the arguments are given in more detail. This analysis is the starting point of part two of this study, where the lattice method will be applied to many of the remaining cases. The results for $4$-primary groups of this part read as follows.

\begin{maintheorem}\label{MainTheorem}
Let $G$ be an almost simple $4$-primary group. The following table shows, whether the HeLP method is sufficient to prove the Prime Graph Question for $G$. A simple group $S$ appearing in bold indicates all almost simple groups having $S$ as a socle. The HeLP method suffices to answer the Prime Graph Question if and only if the isomorphism type of $G$ appears in the left column of the following table. If the HeLP method does not suffice to prove that there are no units of order $pq$ in $\V(\ZZ G)$ for a group $G$ in the right column, then $pq$ is given in italic in parentheses. \\ 
\begin{longtable}{|p{.50\textwidth} | p{.45\textwidth}|} \hline
HeLP method sufficient to prove (PQ) & HeLP method not sufficient to prove (PQ) \\ \hline
$\boldsymbol{\PSL(2,p)}$, $p$ a prime & $\boldsymbol{\PSL(2,2^f)}$, $f \geq 5$ \textit{(6)} \\
{} & $\boldsymbol{\PSL(2,3^f)}$, $f \geq 7$ \textit{(6)} \\ \hline 
$\boldsymbol{A_7}$, $\boldsymbol{A_8}$, $\boldsymbol{A_9}$, $\boldsymbol{A_{10}}$ & {} \\
$\boldsymbol{\PSL(2,25)}$, $\boldsymbol{\PSL(2,49)}$, $\boldsymbol{\PSL(4,3)}$ & $\boldsymbol{\PSL(3,5)}$ \textit{(15)}, $\boldsymbol{\PSL(3,17)}$ \textit{(51)} \\ 
$\boldsymbol{\PSU(3,5)}$, $\boldsymbol{\PSU(3,8)}$, $\boldsymbol{\PSU(3,9)}$ & $\boldsymbol{\PSU(3,7)}$ \textit{(21)} \\
$\boldsymbol{\PSU(4,3)}$, $\boldsymbol{\PSU(4,5)}$, $\boldsymbol{\PSU(4,4)}$, $\boldsymbol{\PSU(5,2)}$ & {} \\
$\boldsymbol{\PSp(4,4)}$, $\boldsymbol{\PSp(4,5)}$, $\boldsymbol{\PSp(4,9)}$, $\boldsymbol{\PSp(6,2)}$ & $\boldsymbol{\PSp(4,7)}$ \textit{(35)} \\
$\boldsymbol{\operatorname{P}\Omega_+(8,2)}$, $\boldsymbol{\Sz(8)}$, $\boldsymbol{G_2(3)}$, $\boldsymbol{{}^3D_4(2)}$  & {} \\
$\boldsymbol{{}^2F_4(2)'}$, $\boldsymbol{M_{11}}$, $\boldsymbol{M_{12}}$, $\boldsymbol{J_2}$ & {} \\

$\PSL(2,16).2$, $\PSL(2,16).4$ & $\PSL(2,16)$ \textit{(6)} \\
$\PSL(2,27).3$, $\PSL(2,27).6$ &  $\PSL(2,27)$ \textit{(6)}, $\PSL(2,27).2$ \textit{(6)} \\
$\PSL(2,81).2a$, $\PSL(2,81).4a$ &  $\PSL(2,81)$ \textit{(6)}, $\PSL(2,81).2b$ \textit{(6, 15)} \\
{} &  $\PSL(2,81).2c$ \textit{(6, 15)}, $\PSL(2,81).2^2$ \textit{(15)} \\
{} & $\PSL(2,81).4b$ \textit{(15)}, $\PSL(2,81).(2\times 4)$ \textit{(15)} \\
{} & $\PSL(2,243)$ \textit{(6)}, $\PGL(2,243)$ \textit{(6, 33)} \\
$\PSL(3,4).2a$, $\PSL(3,4).2b$, $\PSL(3,4).2c$ & $\PSL(3,4)$  \textit{(6)} \\
$\PSL(3,4).3$, $\PSL(3,4).6$, $\PSL(3,4).S_3b$ & {} \\
$\PSL(3,4).S_3c$, $\PSL(3,4).D_{12}$ & {} \\
$\PSL(3,7).3$, $\PSL(3,7).S_3$ & $\PSL(3,7)$ \textit{(21)}, $\PSL(3,7).2$ \textit{(21)} \\
$\PSL(3,8).2$, $\PSL(3,8).3$, $\PSL(3,8).6$ & $\PSL(3,8)$ \textit{(6)} \\
$\PSU(3,4).2$, $\PSU(3,4).4$ & $\PSU(3,4)$ \textit{(6)} \\
$\Sz(32).5$ & $\Sz(32)$ \textit{(10)} \\
\hline 
\end{longtable}
\end{maintheorem}

The proof is given in \cref{results}. Together with the reduction in \cref{reduction} we obtain:
\begin{maincorollary}
Let $G$ be a finite group and assume that all almost simple images of $G$ are $k$-primary with $k$ at most $4$. Assume moreover that if an almost simple image of $G$ appears in the right column of the table in \cref{MainTheorem}, then $G$ contains elements of the orders given in parentheses. Then the Prime Graph Question has an affirmative answer for $G$.
\end{maincorollary}

The class of $4$-primary groups turns out to be a good benchmark for the power of the HeLP method. We encounter several difficulties and limitations when applying it. Firstly the (possibly) infinite series can be handled using generic ordinary and Brauer characters. In many cases we needed to do manual computations, when character tables were not available in the \textsf{GAP} Atlas of Group Representations \cite{AtlasRep} and the \textsf{GAP} character table library \cite{CTblLib}. In these cases we use induced characters, special tables from the literature and other specific arguments. In one situation we encounter a limit of representation theoretical knowledge, namely for the projective symplectic group $\PSp(4,7).$ Here the $2$-Brauer table is only known up to a parameter. We show however that the value of the parameter has no influence on the ability of the HeLP method to answer the Prime Graph Question for this group. 

In the follow-up article the Prime Graph Question will be proved for more of the almost simple groups in question, among others for the series of all $4$-primary $\PSL(2,2^f)$, $\Aut(\PSL(2,81)) = \PSL(2,81).(2\times 4)$, $\PSL(3,4)$ and $\PSU(3,7)$ using the lattice method.

\section{HeLP Method}\label{HeLP_section}

The main notion to study the Zassenhaus Conjecture and related questions are partial augmentations. Let $C$ be a conjugacy class in $G$ and $u = \sum\limits_{g \in G} z_g g \in \ZZ G$. Then 
\[ \varepsilon_C(u) = \sum_{g \in C} z_g \]
is called the partial augmentation of $u$ with respect to $C$. Sometimes we denote it also by $\varepsilon_c(u)$, where $c$ is an element in $C$. Note that for a normalized unit $u$ we have $\sum_{C}\varepsilon_C(u) = 1$, summing over all conjugacy classes $C$ of $G$. The connection between rational conjugacy and partial augmentations is provided by \cite[Theorem 2.5]{MRSW}: A unit $u \in \V(\ZZ G)$ of order $n$ is rationally conjugate to an element of $G$ if and only if $\varepsilon_{C}(u^d) \geq 0$ for all conjugacy classes $C$ in $G$ and all divisors $d$ of $n$. It is known in general that for a torsion unit $u \in \V(\ZZ G)$ we have $\varepsilon_1(u) = 0$, unless $u = 1$, by the Berman-Higman Theorem \cite[Proposition 1.5.1]{EricAngel1}. Moreover if the order of elements in $C$ does not divide the order of $u$, then $\varepsilon_C(u) = 0$ \cite[Theorem 2.3]{HertweckBrauer}. 

A method to study the Zassenhaus Conjecture using ordinary characters of $G$ was introduced by Luthar and Passi in \cite{LutharPassi} and later extended to Brauer characters by Hertweck in \cite{HertweckBrauer}. Subsequently it became known as HeLP (\textbf{He}rtweck\textbf{L}uthar\textbf{P}assi), a name coined by Konovalov. Let $D$ be a $K$-representation of $G$ with character $\chi$ where $K$ is an algebraically closed field of characteristic $p \geq 0$. Then $D$ can be linearly extended to $\ZZ G$ and restricted afterwards to $\V(\ZZ G)$, providing a $K$-representation, also called $D$, of $\V(\ZZ G)$ with the character $\chi$ now defined on all $p$-regular elements of $\V(\ZZ G)$, i.e.\ those torsion units whose order is not divisible by $p$, see \cite[§ 3]{HertweckBrauer}. All torsion units are considered to be $0$-regular. Let $u \in \V(\ZZ G)$ be a torsion unit of order $n$, such that $n$ is not divisible by $p$. Summing over all $p$-regular conjugacy classes $C$ of $G$ we obtain 
\begin{align}\label{CharacterValue}
 \sum_C\varepsilon_C(u) \chi(C) = \chi(u). 
\end{align}
This holds for ordinary characters and also for Brauer characters by \cite[Theorem 3.2]{HertweckBrauer}.

 In the next two paragraphs, let $u \in \V(\ZZ G)$ always be a unit of order $n$.
 
 Denote by $A \sim (\alpha_1,...,\alpha_m)$ the fact that a diagonalizable matrix $A$ has eigenvalues $\alpha_1,...,\alpha_m$ with multiplicities. Now assume one knows, e.g.\ by induction, the eigenvalues of $D(u^d)$ for all divisors $d$ of $n$, apart from $1$. Then one can obtain restrictions on the possible eigenvalues of $D(u)$ in the following way: If $n = q^f$ is a power of the prime $q$ with $D(u^q) \sim (\alpha_1,...,\alpha_m)$ then there exist $\beta_1,...,\beta_m$ such that $\beta_i^q = \alpha_i$ and
\[ D(u) \sim (\beta_1,...,\beta_m). \]
In case $n$ is not a prime power and $q$ and $r$ are different prime divisors of $n$, there are integers $a$ and $b$ such that $aq + br = 1$. Assume that $D(u^{aq}) \sim (\alpha_1,...,\alpha_m)$ and $D(u^{br}) \sim (\beta_1,...,\beta_m)$. Then, since $D(u^{aq})$ and $D(u^{br})$ are simultaneously diagonalizable, there is a permutation $\pi \in S_m$ such that
\[D(u) = D(u^{aq}) D(u^{br}) \sim (\alpha_1\beta_{\pi(1)},...,\alpha_m\beta_{\pi(m)}).\]
Usually the character ring $\ZZ[\chi]$ of $\chi$, i.e. the smallest ring containing the values of $\chi$ on elements of $G$, allows to obtain more information on the eigenvalues of $D(u)$ since $\chi(u)$ has to lie in $\ZZ[\chi]$. Comparing eigenvalues of $D(u)$ obtained in this way with \eqref{CharacterValue} one obtains restrictions on the possible partial augmentations of $u$ and vice versa.

Via discrete Fourier inversion this may be formalized in the following way: Denote by $\zeta$ a fixed primitive complex $n$-th root of unity and by $\xi$ some complex $n$-th root of unity. In case $p > 0$ still $\xi$ can be understood to be a possible eigenvalue of $D(u)$ via a fixed Brauer correspondence of $n$-th roots of unity in $K$ and $\mathbb{C}$, cf. \cite[§ 17]{CR1}. Denote by ${\operatorname{Tr}}_{L/\mathbb{Q}}$ the number theoretical trace of the field extension $L/\mathbb{Q}$. Then the multiplicity $\mu(\xi, u, \chi)$ of $\xi$ as an eigenvalue of $D(u)$, which is a non-negative integer, is
\begin{align}\label{HeLP-restrictions}
\mu(\xi, u, \chi) = \frac{1}{n} \sum_{\substack{d \mid n \\ d \neq 1}} \operatorname{Tr}_{\mathbb{Q}(\zeta^d)/\mathbb{Q}}(\chi(u^d)\xi^{-d}) \ + \frac{1}{n} \sum_{C} \varepsilon_C(u)\operatorname{Tr}_{\mathbb{Q}(\zeta)/\mathbb{Q}}(\chi(C)\xi^{-1}) \in \ZZ_{\geq0}.
\end{align}
Here the second sum runs over all conjugacy classes containing elements of order coprime to $p$, if $p > 0$. Thus, if we know the partial augmentations of $u^d$ for all divisors $d \neq 1$ of $n$, we obtain a system of integral inequalities for the partial augmentations of $u$.

While it is mostly more convenient to calculate with actual eigenvalues of representations, the more formal formulation of the method given in \eqref{HeLP-restrictions} provides an algorithm directly implementable in a computer program which has been done in the HeLP package \cite{HeLPPackage, HeLPPaper} written by the authors. It was programmed for the computer algebra system \textsf{GAP} \cite{GAP} and uses the program \texttt{4ti2} \cite{4ti2} to solve integral inequalities. This package, the character table library of \textsf{GAP} \cite{CTblLib} and the \textsf{GAP} Atlas of Group Representations \cite{AtlasRep} provide the basis of the second part of the current paper. We give the results in a way such that they can easily be reproduced using the package.
On the other hand several results, especially for series of groups, are given in \cref{GeneralResults} with complete proofs. \\

\textit{Notation:} We will use the following notation to indicate that certain eigenvalues occur with greater multiplicity. For a diagonalizable matrix $A$ we write 
\[A \sim \left(\alpha, \beta, \gamma, \eigbox{m}{\delta_1, ..., \delta_k}, \eigbox{n}{\eta_1, ..., \eta_\ell}\right)\]
to indicate that the eigenvalues of $A$ are $\delta_1, ..., \delta_k$ (each with multiplicity $m$), and $\eta_1, ..., \eta_\ell$ (each with multiplicity $n$), $\alpha$, $\beta$ and $\gamma$ (all with multiplicity $1$).

If $C_1,...,C_k$ are all conjugacy classes of elements of some fixed order $n$ in $G$ and $u \in \mathbb{Z}G$, we write
\[\tilde{\varepsilon}_n(u) = \varepsilon_{C_1}(u) + ... + \varepsilon_{C_k}(u).\]

For a conjugacy class $K$ we denote by $K^p$ the conjugacy class containing the $p$-th powers of the elements of $K$. The following Lemma can be found in \cite[Remark 6]{BovdiHertweck}.
\begin{lemma}\label{Wagner}
Let $G$ be a finite group, $u \in \V(\ZZ G)$ a torsion unit and $p$ a prime. Then for every conjugacy class $C$ of $G$
\[ \varepsilon_{C}(u^p) \equiv \sum_{K \, : \, K^p = C} \varepsilon_K(u) \mod p.  \]
Assume $u$ is of order $pq$, where $p$ and $q$ are different primes, and $G$ contains no elements of order $pq$. Then using the Berman-Higman Theorem we obtain 
\[ \tilde{\varepsilon}_p(u) \equiv 0 \mod p \ \ \ \ \ {\text{and}} \ \ \ \ \  \tilde{\varepsilon}_p(u) \equiv 1 \mod q.  \]
\end{lemma}

\begin{remark}\label{Extended_HeLP_restrictions} We provide some explanations about our understanding of the HeLP method.
\begin{enumerate}
\item In view of the criterion for rational conjugacy of torsion units in $\V(\ZZ G)$ to elements of $G$, it is important not only to know the partial augmentations of a torsion unit $u$, but also of its powers $u^d$, where $d$ is a divisor of the order of $u$. For that reason we consider as the \emph{possible partial augmentations for elements of order $n$} the partial augmentations of all $u^d$, for all divisors $d$ of $n$. Since $u^n =1$ always, the partial augmentations of $u^n$ are not included here. Say e.g. a group possesses two conjugacy classes of involutions, $\cc{2a}$ and $\cc{2b}$, and one of elements of order $3$, say $\cc{3a}$, and none of elements of order $6$. Then a typical tuple of possible partial augmentations of a unit $u \in \V(\ZZ G)$ of order $6$ looks like 
\[ (\varepsilon_{\cc{2a}}(u^3), \varepsilon_{\cc{2b}}(u^3), \varepsilon_{\cc{3a}}(u^2), \varepsilon_{\cc{2a}}(u), \varepsilon_{\cc{2b}}(u), \varepsilon_{\cc{3a}}(u)). \]
As above, we always list only those partial augmentations of units which might not be equal to $0$.\\
We call a tuple of possible partial augmentations of $u \in \V(\ZZ G)$ \emph{trivial}, if it coincides with the tuple of partial augmentations of an element $g \in G$. By \cite[Theorem 2.5]{MRSW} this is the case if and only if this tuple consists of non-negative numbers.
\item When we speak of the maximal possible results using the HeLP method, we include into this all results listed so far in this paragraph. This means if $u \in \V(\ZZ G)$ is a torsion unit of order $n \not= 1$, then
\begin{itemize}
 \item $\varepsilon_1(u) = 0$ and $\varepsilon_C(u) = 0$, if a conjugacy class $C$ consists of elements of order not dividing $n$.
 \item $u$ satisfies the HeLP constraints given in \eqref{HeLP-restrictions} for all ordinary characters of $G$ and all $p$-Brauer characters of $G$, for all primes $p$ not dividing $n$.
 \item $u$ satisfies \cref{Wagner}. 
\end{itemize}
\item Let $u \in \V(\ZZ G)$ be a torsion unit of order $n$ and $q$ and $r$ different prime divisors of $n$. Assume that $D$ is a representation having a character which is rational valued on all conjugacy classes of elements of order dividing $n$. By the method explained above we have $D(u) = D(u^{aq}) D(u^{br})$ for some integers $a$ and $b$. For every fixed $k$, the multiplicity of all primitive $k$-th roots of unity as eigenvalues of $D(u^{aq})$ and of $D(u^q)$ have to be the same, as their trace is rational. Hence $D(u^{aq}) \sim D(u^q)$ and we will give in such situations the diagonal form of $D(u^q)$ instead of the diagonal form of $D(u^{aq})$.
\end{enumerate}
\end{remark}

It is in general much harder to compute Brauer tables of a group than ordinary character tables. For this reason it is often of importance to be able to exclude some prime divisors $p$ of the order of $G$ as candidates whose $p$-Brauer tables will provide new information for the HeLP method. If a $p$-Brauer character coincides on all $p$-regular conjugacy classes of $G$ with an ordinary character it is called \textit{liftable}, cf.\ e.g.\ \cite{White2}, and will surely not provide new information. E.g.\ by the Fong-Swan-Rukolaine Theorem \cite[Theorem 22.1]{CR1} any $p$-Brauer character of a $p$-solvable group is liftable and one does not need to include $p$-Brauer tables in the computations in this case.

\section{General Results}\label{GeneralResults}

In this section we prove general results, independent of the special class of $4$-primary groups. These results turn out however to be very useful for this class of groups.
\begin{lemma}\label{Lemmaxy}
Let $G$ be a finite group, $N$ a normal subgroup and $p$ a prime. Assume $N$ contains exactly one conjugacy class of elements of order $p$, the $p$-rank of $G/N$ is 1 and that $G/N$ possesses a representation $D$ over a field of characteristic coprime to $p$ such that a $p$-element of $G/N$ is in the center, but not in the kernel of $D$.  Assume moreover, that any two groups of order $p$ which do not lie in $N$ are conjugate in $G$. Then units of order $p$ in $\V(\mathbb{Z}G)$ are rationally conjugate to elements of $G$.
\end{lemma}

\begin{proof} Let $u \in \V(\mathbb{Z} G)$ be a unit of order $p$. Let $C$ be the conjugacy class of elements of order $p$ contained in $N$. If this is the only class of elements of order $p$ in $G$, then we are done. So assume there are classes $K_1, ..., K_{p-1}$ of elements of order $p$ in $G \setminus N$ such that $K_1^j = K_j$, this is possible by the assumptions on $G$. We have 
\begin{align}\label{LemmaxySumme}
\varepsilon_C(u) + \varepsilon_{K_1}(u) + ... + \varepsilon_{K_{p-1}}(u) = 1.
\end{align} 
Denote by $\pi \colon G \to G/N$ the natural projection and set $D' = D \circ \pi$, a representation of $G$ of degree $d$, say. Let $\chi$ be the character afforded by $D'$. Let $\chi(K_1) = d\zeta$ for some primitive $p$-th root of unity $\zeta$. We have
\begin{align*} 
\chi(u) & = d\varepsilon_{C}(u) + d\zeta\varepsilon_{K_1}(u) + ... + d\zeta^{p-1}\varepsilon_{K_{p-1}}(u) \\ & = d(\zeta(\varepsilon_{K_1}(u) - \varepsilon_{C}(u)) + ... + \zeta^{p-1}(\varepsilon_{K_{p-1}}(u) - \varepsilon_{C}(u))). 
\end{align*}
Hence $\chi(u)/d  \in \mathbb{Z}[\zeta]$. On the other hand, $\chi(u)$ is the sum of $d$ roots of unity and hence $\chi(u) = d\zeta^{i}$ for some $i$. Thus
\[ \zeta^i = \zeta( \varepsilon_{K_1}(u) - \varepsilon_C(u)) + ... + \zeta^{p-1} (\varepsilon_{K_{p-1}}(u) - \varepsilon_C(u)). \]
Considering the basis $\zeta,...,\zeta^{p-1}$ of $\mathbb{Z}[\zeta]$ and using \eqref{LemmaxySumme} we obtain that exactly one of the partial augmentations $\varepsilon_C(u)$, {$\varepsilon_{K_1}(u),...,\varepsilon_{K_{p-1}}(u)$} must be 1 while all others are $0$.
\end{proof}

We will apply this lemma frequently for representations of degree $1$ and refer to it in \cref{HeLPTable} along with the character inflated to $G$.

To establish the results for the almost simple groups containing $\PSL(2, p^f)$ we will need some well-known properties of these groups and their representation theory which we collect for the convenience of the reader in the following remark.

\begin{remark}\label{Gruppentheorie}
Let $G = \PGL(2,q)$ and $H = \PSL(2,q)$, where $q = p^f$. Set $d = \gcd(2,p-1)$. Then $|G| = (q - 1)q(q + 1)$ and $|H| = |G|/d$. 
\begin{enumerate} \item\label{PSL_orders} The orders of elements of $G$ are exactly the divisors of $ q - 1, q + 1$ and $p$, while the orders of elements in $H$ are exactly the divisors of $(q - 1)/d, (q + 1)/d$ and $p$. The group $H$ possesses a partition into cyclic subgroups, so different cyclic subgroups of the same order have trivial intersection. There are $d$ conjugacy classes of elements of order $p$ in $H$ and these classes fuse in $G$. If $g \in H$ is of order coprime to $p$ then $g^{-1}$ is the only conjugate of $g$ in $\langle g \rangle$ and the conjugacy class of $g$ is the same in $G$ \cite[Chapter XII]{Dickson} (see also \cite[II, §6 - §8]{HuppertI}). 

\item\label{PSL_autos} The outer automorphism group of $H$ is isomorphic to $C_d \times C_f$ while the outer automorphism group of $G$ is isomorphic to $C_f$. The group $C_f$ is induced by the projection of the entry-wise action of the Frobenius automorphism of $\mathbb{F}_{q}$ on the $2\times 2$-matrices over $\mathbb{F}_{q}$ \cite[3.3.4]{Wilson}. Let $\sigma$ be a generator of this group $C_f$. The group $\langle H, \sigma\rangle$ is called $\operatorname{P \Sigma L}(2,q)$. The full automorphism group of $H$ is $\operatorname{P \Gamma L}(2,q)$, the group of semilinearities of the projective line over $\mathbb{F}_{q}$.  

\item\label{PSL_characters} The natural permutation representation of $H$, $G$ and $\operatorname{P \Gamma L}(2,q)$ on the projective line over $\mathbb{F}_q$, which has $q + 1$ elements, is 2-transitive \cite[XI, Example 1.3]{HuppertIII}. Thus modulo the trivial representation we obtain an irreducible ordinary representation of degree $q$ \cite[V, Satz 20.2]{HuppertI}. The character corresponding to this representation is known as the Steinberg character. With the facts given in \cite[XI, Example 1.3]{HuppertIII} the character values may easily be calculated to be:

\[ \psi(x) = \begin{cases} q & x = 1 \\ 1 & o(x) \mid \frac{q - 1}{d} \\ 0 & o(x) \mid p \\ -1 & o(x) \mid \frac{q + 1}{d}. \end{cases}   \]

The ordinary character tables of $G$ and $H$ were first computed in \cite{Schur} and \cite{Jordan}. In particular for every prime $r$ dividing $\frac{q - 1}{d}$ and fixed conjugacy class $C$ of elements of order $r$ in $H$ there exist characters $\chi_i$ of degree $q + 1$ such that $\chi_i(C) = \zeta^i + \zeta^{-i}$ for some fixed primitive $r$-th root of unity $\zeta$ and all possible $i$. Moreover $\chi_i(g) = 1$, if $g$ has order $p$, and $\chi_i(g) = 0$, if the order of $g$ divides $\frac{q + 1}{d}$. 

The general linear group $\operatorname{GL}(2,q)$ acts via conjugation on the Lie Algebra $\mathfrak{sl}_2(\FF_q)$ of $2 \times 2$-matrices over $\FF_q$ with trace $0$. The kernel of this action is exactly the center of $\operatorname{GL}(2,q)$ giving a 3-dimensional $\FF_q$-representation of $G$, say $P$ with character $\varphi$. Let $g$ be an element of $G$ of order $r$ coprime to $p$. Computing the eigenvalues of $P(g)$ we find $\varphi(g) = 1 + \zeta + \zeta^{-1}$ for some primitive $r$-th root of unity $\zeta$.

Let $r$ be an odd prime divisor of $|H|$ different from $p$. By \cite{Burkhardt} any $r$-Brauer character of $H$ is liftable and any $2$-Brauer character of $H$ is liftable if and only if $q \equiv -1 \mod 4$. If $p$ is odd, there are exactly two irreducible ordinary characters of $H$ whose inductions to $G$ are irreducible and the induced character $\eta$ is the same for both. Since both these $H$-characters take different values on the conjugacy classes of elements of order $p$ the reduction of $\eta$ modulo $r$ will be irreducible. This implies that any $r$-modular Brauer character of $G$ is also liftable. 
\end{enumerate}
\end{remark}

A proof of the following results can be found in \cite[Propositions 6.3, 6.4, 6.7]{HertweckBrauer}.

\begin{proposition}[Hertweck]\label{HertweckPSL} Let $G = \PSL(2, p^f)$ and $u \in \V(\ZZ G)$ a torsion unit. Then the following statements hold:
\begin{enumerate}
 \item If $u$ is of prime order different from $p$, then $u$ is rationally conjugate to an element of $G$.
 \item If the order of $u$ is not divisible by $p$, then the order of $u$ coincides with the order of an element of $G$.
 \item If $f = 1$, then (PQ) has an affirmative answer for $G$.
\end{enumerate}
\end{proposition}

With a generalization of Hertweck's proofs of \cite[Propositions 6.3, 6.7]{HertweckBrauer} we obtain the following result for $\PGL(2,p^f)$.

\begin{proposition}\label{PGL2} 
Let $G = \PGL(2,p^f)$. If $u \in \V(\ZZ G)$ is a torsion unit of order coprime to $p$, then $u$ has the same order as an element of $G$. If $f = 1$, the Prime Graph Question has an affirmative answer for $G$ and thus for any almost simple group with socle $\PSL(2,p)$.
\end{proposition}

\begin{proof}
 The Zassenhaus Conjecture for small values of $p^f$ was proved in \cite{HughesPearson} for $\PGL(2,2) \simeq S_3$, in \cite{AllenHobbyS4} for $\PGL(2,3) \simeq S_4$, in \cite{LutharPassi} for $\PGL(2,4) \simeq A_5$ and in \cite{LutharTrama} for $\PGL(2,5) \simeq S_5$. (ZC) for these groups may be also easily verified using HeLP. 

Let first $f = 1$ and let $u \in \V(\ZZ G)$ be of order $pr$ for a prime $r$ different from $p$. Then there are no elements of order $pr$ in $G$.
Let $\eta$ be the twist of the Steinberg character with the non-trivial character of degree 1 and let $D$ be a representation affording $\eta$. Then $\eta$ is integral and takes the same value on both classes of involutions of $G$. First assume $r \mid p - 1$. Then
 \begin{align*}
 D(u^r) &\sim \left(1, \zeta_p,...,\zeta_p^{p-1} \right), \\
 D(u^p) &\sim \left(\eigbox{\frac{p+r-1}{r}}{1}, \eigbox{\frac{p-1}{r}}{\zeta_r,\zeta_r^2,...,\zeta_r^{r-1}} \right)
\end{align*}
and the sum of eigenvalues of $D(u)$ can not be an integer, since we may assume $p > 3$.\\
Next assume $r \mid p + 1$. Then
 \begin{align*}
 D(u^r) &\sim \left(1, \zeta_p,...,\zeta_p^{p-1} \right), \\
 D(u^p) &\sim \left(\eigbox{\frac{p -r + 1}{r}}{1}, \eigbox{\frac{p + 1}{r}}{1, \zeta_r,\zeta_r^2,...,\zeta_r^{r-1}} \right)
\end{align*}
and again the sum of eigenvalues of $D(u)$ is not an integer, since $p > 3$. 

So let $G =\PGL(2,p^f)$ and let $u \in \V(\ZZ G)$ be of order coprime to $p$ so that the order of $u$ is different from the order of all the elements of $G$. Let $D$ be the $3$-dimensional $p$-modular representation of $G$ described in \cref{Gruppentheorie}\eqref{PSL_characters} with character $\varphi$.  Assume first that the order of $u$ is odd, by \cref{Gruppentheorie} we hence may assume that $u$ is of order $rs$ where $r$ and $s$ are odd primes. Let $g$ and $h$ be elements of $G$ of order $r$ and $s$ respectively. Since
\[ \varphi(u^s) = \sum_{i = 1}^{(r-1)/2} \varepsilon_{g^i}(u^s) (1 + \zeta_r^i + \zeta_r^{-i}) = 1 + \sum_{i = 1}^{(r-1)/2} \varepsilon_{g^i}(u^s) (\zeta_r^i + \zeta_r^{-i}) \]
is the sum of exactly three $r$-th roots of unity and $\sum_{i = 1}^{(r-1)/2} \varepsilon_{g^i}(u^s) = 1$, we obtain that exactly one $\varepsilon_{g^i}(u^s)$ is $1$, while all others are $0$ and $u^s$ is rationally conjugate to an element of $G$. The same applies for $u^r$. Thus there exists a primitive $rs$-th root of unity $\xi$ such that
\[D(u^r) \sim (1,\xi^r, \xi^{-r}) \qquad \text{and} \qquad D(u^s) \sim (1,\xi^s, \xi^{-s}). \]
Since $\varphi(u)$ is real, $D(u)$ has the eigenvalue 1 and two primitive $rs$-th roots of unity as eigenvalues which are inverses of each other.  Thus there exists some $k$ coprime with $rs$ such that
\begin{align*}
\varphi(u) = 1 + \xi^k + \xi^{-k} &= \sum_{i = 1}^{(r-1)/2} \varepsilon_{g^i}(u) (1 + \zeta_r^i + \zeta_r^{-i}) + \sum_{i = 1}^{(s-1)/2} \varepsilon_{h^i}(u) (1 + \zeta_s^i + \zeta_s^{-i}) \\
&= 1 + \sum_{i = 1}^{(r-1)/2} \varepsilon_{g^i}(u) (\zeta_r^i + \zeta_r^{-i}) + \sum_{i = 1}^{(s-1)/2} \varepsilon_{h^i}(u) (\zeta_s^i + \zeta_s^{-i}).
\end{align*}
Hence
\begin{align}\label{xi}
\xi^k + \xi^{-k} = \sum_{i = 1}^{(r-1)/2} \varepsilon_{g^i}(u) (\zeta_r^i + \zeta_r^{-i}) + \sum_{i = 1}^{(s-1)/2} \varepsilon_{h^i}(u) (\zeta_s^i + \zeta_s^{-i}). 
\end{align} 
Consider $\{\xi^j \ | \ \gcd(j,rs) = 1\}$ as a basis of $\mathbb{Z}[\xi]$. Note that a primitive $r$-th root of unity has coefficient sum $-(s-1)$ with respect to this basis, while a primitive $s$-th root of unity has coefficient sum $-(r-1)$. Recall that $\tilde{\varepsilon}_r(u)$ denotes the sum of all partial augmentations of $u$ at classes with elements of order $r$. Then comparing coefficient sums in \eqref{xi} we obtain
\[ 2 = -2(s-1) \sum_{i = 1}^{(r-1)/2} \varepsilon_{g^i}(u) - 2(r-1)\sum_{i = 1}^{(s-1)/2} \varepsilon_{h^i}(u) = -2(s-1)\tilde{\varepsilon}_r(u) - 2(r-1)\tilde{\varepsilon}_s(u).\]
Using $\tilde{\varepsilon}_r(u) + \tilde{\varepsilon}_s(u) = 1$ this implies $-1 = (s-1)(1-\tilde{\varepsilon}_s(u)) + (r - 1)\tilde{\varepsilon}_s(u)$ and so $\tilde{\varepsilon}_s(u) = \frac{-s}{r-s}$. A contradiction, since $r$ and $s$ are odd.  

Finally, assume the order $u$ is even. By \cref{Gruppentheorie} we can assume that $u$ is of order $4r$ where $r$ is an odd prime dividing $q-1$, if $4$ divides $q+1$, and $r$ divides $q+1$, if $4$ divides $q-1$. By \cite[Lemma 3.1]{SIP} we can assume that $u^r$ is rationally conjugate to an element of $G$. Hence
\[D(u^r) \sim (1,\zeta_4, \zeta_4^{-1}) \qquad \text{and} \qquad D(u^4) \sim (1,\zeta_r, \zeta_r^{-1}) \]
for a certain primitive $r$-th root of unity $\zeta_r$. As $\varphi$ is real and $\zeta_4^{-1} = -\zeta_4$, we get that $\varphi(u) = 1 + \zeta_4(\zeta_r^i - \zeta_r^{-i})$ for some integer $i$ coprime to $r$. Hence $\varphi(u)$ is an element of $\mathbb{Q}(\zeta_{4r})$ which does not lie in $\mathbb{Q}(\zeta_r)$. On the other hand the orders of elements in $G$ which do divide $4r$ are $2$, $4$, $r$ and $2r$. On all of these elements $\varphi$ takes a value in $\mathbb{Q}(\zeta_r)$, so no integral linear combination of them can ever land outside of this field. We conclude that the value we computed for $\varphi(u)$ is impossible.
\end{proof}

\begin{theorem}\label{PSL_2_p2}
Let $G$ be an almost simple group with socle $S = \PSL(2,p^2)$ for some prime $p$. Then the Prime Graph Question has an affirmative answer for $G$.
\end{theorem}

\begin{proof}
If $p = 2$, then $G$ is the alternating or symmetric group of degree $5$. In that case even the Zassenhaus Conjecture holds for $G$ \cite{LutharPassi, LutharTrama}. If $p = 3$, then $S$ is the alternating group of degree $6$ and the Prime Graph Question has been answered for all possibilities of $G$ \cite{HertweckA6, KonovalovKimmiStAndrews, Gitter}. So assume $p \geq 5$.\\
We will use the facts given in \cref{Gruppentheorie}\eqref{PSL_autos}. Let $g$ be an involution in $\PGL(2,p^2)$ not lying in $\PSL(2,p^2)$ and let $\sigma$ be the automorphism of $S$ induced by the Frobenius automorphism of $\mathbb{F}_{p^2}$. Then the outer automorphism group of $S$ is generated by $g$ and $\sigma$, modulo the inner automorphisms, and is isomorphic to the Klein four group, see \cite[XI, Example 1.3]{HuppertIII}. We will use the names for the groups given in \cref{normal_subgroups_AutPSL2p2} where $\PGL(2,p^2) = \langle \PSL(2,p^2), g \rangle$ and $\operatorname{P \Sigma L}(2,p^2) = \langle \PSL(2,p^2), \sigma \rangle$. 
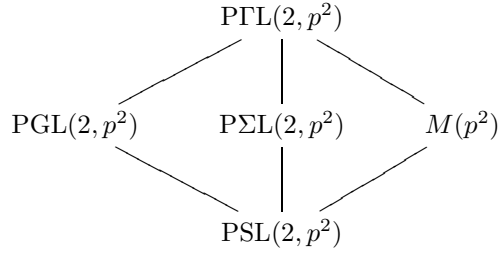
\begin{figure}[ht]
\caption{Almost simple groups containing $\PSL(2,p^2)$, indices of successive groups are $2$.}
\label{normal_subgroups_AutPSL2p2}
\begin{equation*}
  \xymatrix{
  & \operatorname{P \Gamma L}(2,p^2) \ar@{-}[dl] \ar@{-}[d] \ar@{-}[dr] & \\
  \PGL(2,p^2) \ar@{-}[dr] & \operatorname{P\Sigma L}(2,p^2) \ar@{-}[d] & M(p^2) \ar@{-}[dl] \\
  & \operatorname{PSL}(2,p^2) & }
\end{equation*}
\end{figure}

Let $q$ and $r$ be different primes dividing the order of $S$.
%Then $S$ contains no elements of order $qr$ if and only if either $q$ or $r$ equals $p$ or $q \mid \frac{p^2-1}{2}$ and $r \mid \frac{p^2+1}{2}$.
Then $S$ contains elements of order $qr$ if and only if $q$ and $r$ both divide $\frac{p^2-1}{2}$ or both divide $\frac{p^2+1}{2}$.
Moreover $\PGL(2,p^2)$ contains an element of order $2q$ for any $q \neq p$, cf. \cref{Gruppentheorie}\eqref{PSL_orders}, and $\operatorname{P \Sigma L}(2,p^2)$ contains elements of order $2p$ and $2q$ whenever $q \mid \frac{p^2-1}{2}$, since $\sigma$ commutes exactly with the elements lying in the natural subgroup $\PSL(2,p)$. Note that if $C_1$ and $C_2$ are different conjugacy classes in $\PSL(2,p^2)$ of elements of prime order $q \neq p$, then they are still different classes in $\PGL(2,p^2)$ by \cref{Gruppentheorie}\eqref{PSL_orders}. They fuse into one class in $\operatorname{P \Sigma L}(2,p^2)$, $M(p^2)$ and $\operatorname{P \Gamma L}(2,p^2)$ if and only if the $p$-th powers of elements in $C_1$ lie in $C_2$, in particular $q \mid \frac{p^2 + 1}{2}$. This can be seen when thinking of the elements of $C_1$ as diagonalized matrices (if necessary over a bigger field). Then $\sigma$ is conjugating such an element to its $p$-th power. Note that all elements of prime order in $M(p^2)$ already lie in $S$. This is clear for elements of odd order and follows for involutions from \cite[End of proof of Lemma 2.3]{GorensteinCentralizers}.

Let $q \mid \frac{p^2 + 1}{2}$ and $2 \not= r \mid \frac{p^2 - 1}{2}$ be primes. By the discussion above it suffices to prove that the following units $u \in \V(\ZZ G)$ do not exist:
\begin{itemize}
\item $G = \operatorname{P \Gamma L}(2,p^f)$ and $u$ has order $qr$, $pq$ or $pr$.
\item $G = M(p^2)$ and $u$ has order $2p$ or $2q$.
\item $G = \PGL(2,p^2)$ and $u$ has order $2p$.
\item $G = \operatorname{P \Sigma L}(2,p^2)$ and $u$ has order $2q$.
\end{itemize}

We will consider all these cases, some of them simultaneously.

First let $G = \operatorname{P \Gamma L}(2,p^2)$ and $q \mid \frac{p^2 + 1}{2}$ and $r \mid \frac{p^2 - 1}{2}$ be odd primes. Let $\psi$ be the $p$-modular Brauer character of $G$ of degree $6$ induced from the Brauer character $\varphi$ of $\PGL(2,p^2)$ having degree $3$ described in \cref{Gruppentheorie}\eqref{PSL_characters}. Let $D$ be a representation affording $\psi$. 
We will show first that units of order $r$ and $q$ in $\V(\ZZ G)$ are rationally conjugate to elements of $G$. If $v \in \V(\ZZ G)$ has order $r$ and $g$ is an element of order $r$ in $G$ then $g$ lies in $\PGL(2,p^2)$ and $\psi(g)$ is just twice the value of $\varphi(g)$. We thus can argue in the same way as in the proof of \cref{PGL2} to prove that $v$ is rationally conjugate to an element of $G$. If $v$ is of order $q$, let $g$ be an element of order $q$ in $G$. Then there exists a primitive $q$-th root of unity $\xi$ such that $D(g) \sim (1,\xi, \xi^{-1}, 1, \xi^{p}, \xi^{-p})$. Let $C_1,...,C_k$ be the conjugacy classes of elements of order $q$ in $G$ such that $g^i \in C_i$ for all $i$. Then
\[\psi(v) = 2 + \sum_{i}\varepsilon_{C_i}(v) \left(\xi^i + \xi^{-i} + \xi^{ip} + \xi^{-ip}\right). \]
Viewing $\{\xi^i + \xi^{-i} + \xi^{ip} + \xi^{-ip} \ | \ (i,q) = 1\}$ as a basis of $\ZZ[\xi + \xi^{-1} + \xi^{p} + \xi^{-p}]$, using the fact that $\sum_i\varepsilon_{C_i}(v) = 1$ and that $\psi(v)$ is the sum of exactly $6$ roots of unity, we obtain that $v$ is rationally conjugate to an element in $G$.

Assume now that $u \in \V(\ZZ G)$ is a unit of order $qr$. Let  $\zeta$ be a primitive $qr$-th root of unity. For $g \in G$ of order $r$ we have $\psi(g) = 2(1 + \zeta^{qi} + \zeta^{-qi})$ for some $i$ coprime to $r$ and for $h \in G$ of order $q$ we have $\psi(h) = 2 + \zeta^{rj} + \zeta^{-rj} + \zeta^{prj} + \zeta^{-prj}$ for some $j$ coprime to $q$. Consider $B = \{\zeta^i + \zeta^{-i} \ | \ (i,qr) = 1 \}$ as a $\mathbb{Z}$-basis of $\ZZ[\zeta + \zeta^{-1}]$. Then by the above and using $\tilde{\varepsilon}_r(u) + \tilde{\varepsilon}_q(u) = 1$ the coefficient sum of $\psi(u)$ expressed with respect to the basis $B$ is exactly
\begin{align}
\tilde{\varepsilon}_r(u)&\left((q-1)(r-1) - 2(q-1)\right) + \tilde{\varepsilon}_q(u)((q-1)(r-1) - 2(r-1))  \nonumber \\
=& (q-1)(r-1) - 2(q-1) \tilde{\varepsilon}_r(u) - 2(r-1)(1 - \tilde{\varepsilon}_r(u)) \nonumber \\
=& (q-3)(r-1) + 2(r-q)\tilde{\varepsilon}_r(u). \label{CoeffSum}
\end{align}
On the other hand from the fact that $u^r$ and $u^q$ are rationally conjugate to elements of $G$ we conclude that there exists an $i$ coprime to $r$ and a $j$ coprime to $q$ such that
\[D(u^q) \sim (1,1,\zeta^{qi},\zeta^{-qi}, \zeta^{qi}, \zeta^{-qi}), \ \ \ \ D(u^r) \sim (1,1,\zeta^{rj}, \zeta^{-rj}, \zeta^{prj}, \zeta^{-prj}).\]
Then $\psi(u)$ is either the sum of $2$ and four primitive $qr$-th roots of unity or the sum of two primitive $r$-th roots of unity, two primitive $q$-th roots of unity and two primitive $qr$-th roots of unity. Comparing the corresponding coefficient sum expressed with respect to $B$ to the sum computed in \eqref{CoeffSum} we find in the first case
\[(q-1)(r-1) + 2 = (q - 3)(r-1) + 2(r-q)\tilde{\varepsilon}_r(u)\ \  \Leftrightarrow\ \ \tilde{\varepsilon}_r(u) = \frac{r}{r-q}, \]
a contradiction. In the second case we obtain
\[-(q-1) -(r-1) + 1 = (q - 3)(r-1) + 2(r-q)\tilde{\varepsilon}_r(u). \]
The left hand side of this equation is odd, while the right side is even, so $u$ does not exist.

Now assume that $u \in \V(\ZZ G)$ is of order $pq$, for a prime $q \mid \frac{p^2 + 1}{2}$. Denote by $C$ the conjugacy class of elements of order $p$ in $G$ and let $\zeta$ be a primitive $pq$-th root of unity. We will use the ordinary character $\chi_1$ of $\PSL(2, p^2)$ of degree $p^2 + 1$ described in \cref{Gruppentheorie}\eqref{PSL_characters}. It takes the value $0$ on each conjugacy class of elements of order $q$, if $q \mid \frac{p^2+1}{2}$, and the sum of two primitive $r$-th roots of unity, if $r \mid \frac{p^2-1}{2}$. Recall that $\chi_1(g) = 1$ for elements of order $p$. Let $\chi$ be the character induced from $\chi_1$ on $G$. Computing several multiplicities of eigenvalues of a representation realizing $\chi$ using \eqref{HeLP-restrictions} we obtain
\begin{align*}
%\mu(\zeta, u, \chi) &= \frac{1}{pq}(4(p^2 + 1) + \Tr_{\QQ(\zeta^q)/\QQ}(4\zeta^{-q}) + \varepsilon_C(u)\Tr_{\QQ(\zeta)/\QQ}(4\zeta^{-1}) = \frac{4}{pq}(p^2 + \varepsilon_C(u)) \\
\mu(\zeta^p, u, \chi) &= \frac{1}{pq}\left(4(p^2 + 1) + \Tr_{\QQ(\zeta^q)/\QQ}(4) + \varepsilon_C(u)\Tr_{\QQ(\zeta)/\QQ}(4\zeta^{-p})\right) \\ &= \frac{4}{pq}\left(p^2 + p - (p - 1)\varepsilon_C(u)\right), \\[.5cm]
\mu(1, u, \chi) &= \frac{1}{pq}\left(4(p^2 + 1) + \Tr_{\QQ(\zeta^q)/\QQ}(4) + \varepsilon_C(u)\Tr_{\QQ(\zeta)/\QQ}(4)\right) \\ &= \frac{4}{pq}\left(p^2 + p + (p - 1)(q - 1)\varepsilon_C(u)\right) 
\end{align*} 
Since $\varepsilon_C(u) \equiv 0 \mod p$ by \cref{Wagner} and $p \geq 5$ these multiplicities imply $\varepsilon_C(u) \in \{0, p\}$, since otherwise one of the multiplicities is negative. But since also $\varepsilon_C(u) \equiv 1 \mod q$ by \cref{Wagner} we get $q \mid \frac{p^2 - 1}{2}$, contradicting our assumption on $q$. 

So assume $u \in \V(\ZZ G)$ is a unit of order $pr$ with an odd prime $r \mid \frac{p^2 - 1}{2}$. Denote by $\chi$ the character of $G$ obtained from inducing the sum of all Galois conjugates of $\chi_1$. Thus $\chi$ has degree $4(p^2 + 1) \cdot \frac{r-1}{2}$, $\chi(g) = -4$, if $g$ is of order $r$, and $\chi(h) = 4\cdot \frac{r - 1}{2}$ for every $h$ of order $p$. Using again \eqref{HeLP-restrictions} and $\varepsilon_C(u) + \tilde{\varepsilon}_r(u) = 1$ we obtain
\begin{align*}
\mu(\zeta^p, u, \chi) =& \frac{1}{pr}\Big(4(p^2 + 1)\frac{r - 1}{2} + \Tr_{\QQ(\zeta^r)/\QQ}\left(4\frac{r - 1}{2}\right) + \Tr_{\QQ(\zeta^p)/\QQ}(-4\zeta^p) \\ &+ \varepsilon_C(u)\Tr_{\QQ(\zeta)/\QQ}\left(4\frac{r - 1}{2}\zeta^{-p}\right) + \tilde{\varepsilon}_r(u)\Tr_{\QQ(\zeta)/\QQ}(-4\zeta^{-p}) \Big) \\
 =& \frac{4}{pr}\left(p(p+1)\frac{r - 1}{2} + p - (\frac{r + 1}{2})(p - 1)\varepsilon_C(u)\right), \\[.5cm]
\mu(1, u, \chi) =& \frac{1}{pr}\bigg(4(p^2 + 1)\frac{r - 1}{2} + \Tr_{\QQ(\zeta^r)/\QQ}\left(4\frac{r - 1}{2}\right) + \Tr_{\QQ(\zeta^p)/\QQ}(-4) \\ &+ \varepsilon_C(u)\Tr_{\QQ(\zeta)/\QQ} \left(4\frac{r - 1}{2}\right) + \tilde{\varepsilon}_r(u)\Tr_{\QQ(\zeta)/\QQ}(-4) ) \bigg) \\ 
=& \frac{4}{pr}\left(p(p+1)\frac{r - 1}{2} - p(r - 1) + (\frac{r^2 - 1}{2})(p - 1)\varepsilon_C(u)\right). 
\end{align*} 
Since $p \geq 5$ and $\varepsilon_C(u) \equiv 0 \mod p$ and $\varepsilon_C(u) \equiv 1 \mod r$ not both these multiplicities can simultaneously be non-negative integers, completing the proof of the Prime Graph Question for $G = \operatorname{P \Gamma L}(2,p^2)$.

Assume next that $G = M(p^2)$ or $G = \PGL(2,p^2)$ and $u$ has order $2p$. We will use the character $\chi$, which is the induced character of the character $\chi_1$ of degree $p^2 + 1$ of $S$ described in \cref{Gruppentheorie}\eqref{PSL_characters}. Denote by $\cc{2a}$ the conjugacy class of involutions in $G$ which lie already in $S$ and by $\cc{2b}$ the other class of involutions in $G$, if $G = \PGL(2,p^2)$. If $ G= M(p^2)$, the unique class of involutions is contained in $S$ and the class $\cc{2b}$ can just be ignored in the following computations (see the discussion following \cref{normal_subgroups_AutPSL2p2}). Denote by $C$ again the class of elements of order $p$. Then $\chi(\cc{2a}) = -4, \chi(C) = 2$ and $\chi(\cc{2b}) = 0$. There is a character $\tau$ containing $S$ in its kernel and mapping elements outside of $S$ to $-1$. By \cref{Lemmaxy} $u^p$ is then rationally conjugate to an element of $G$. Assume first that $u^p$ is rationally conjugate to an element of $\cc{2a}$. From $\tau$ we then get
\[\varepsilon_{\cc{2a}}(u) + \varepsilon_C(u) - \varepsilon_{\cc{2b}}(u) = 1\]
and since $u$ is normalized this implies $\varepsilon_{\cc{2b}}(u) = 0$ and $\varepsilon_{\cc{2a}}(u) = 1 - \varepsilon_C(u)$. Denoting by $\zeta$ a primitive $2p$-th root of unity and computing again multiplicities of certain roots of unity as eigenvalues of $u$ under a representation realizing $\chi$ we obtain 
\begin{align*}
\mu(-1, u, \chi) =& \frac{1}{2p}\Big(2(p^2 + 1) + \Tr_{\QQ/\QQ}((-4)\cdot (-1)) + \Tr_{\QQ(\zeta^2)/\QQ}(2) \\ 
&+ \varepsilon_C(u)\Tr_{\QQ(\zeta)/\QQ}(2\cdot (-1)) + \varepsilon_{\cc{2a}}(u)\Tr_{\QQ(\zeta)/\QQ}(-4\cdot (-1) ) \Big) \\ 
=& \frac{1}{p}\left(p^2 + 3p -3(p - 1)\varepsilon_C(u)\right), \\[.5cm]
\mu(1, u, \chi) =& \frac{1}{2p}\Big(2(p^2 + 1) + \Tr_{\QQ/\QQ}(4) + \Tr_{\QQ(\zeta^2)/\QQ}(2) \\ 
&+ \varepsilon_C(u)\Tr_{\QQ(\zeta)/\QQ}(2) + \varepsilon_{\cc{2a}}(u)\Tr_{\QQ(\zeta)/\QQ}(-4)  \Big) \\ 
=& \frac{1}{p}\left(p^2 - p + 3(p - 1)\varepsilon_C(u) \right). 
\end{align*} 
Since $p \geq 5$ this contradicts $\varepsilon_C(u) \equiv 0 \mod p$ and $\varepsilon_C(u) \equiv 1 \mod 2$ and the fact that both expressions are non-negative. 

So assume $u^p$ is rationally conjugate to elements in $\cc{2b}$. Then form $\tau$ we get $\varepsilon_{\cc{2b}}(u) = 1$ and $\varepsilon_{\cc{2a}}(u) = -\varepsilon_C(u)$. Computing the multiplicities as above we obtain
\begin{align*}
\mu(-1, u, \chi) =& \frac{1}{2p}\Big(2(p^2 + 1) + \Tr_{\QQ(\zeta^2)/\QQ}(2) \\
&+ \varepsilon_C(u)\Tr_{\QQ(\zeta)/\QQ}(2\cdot (-1)) + \varepsilon_{\cc{2a}}(u)\Tr_{\QQ(\zeta)/\QQ}(-4\cdot (-1) ) \Big) \\ 
=& \frac{1}{p}(p^2 + p -3(p - 1)\varepsilon_C(u)), \\[.5cm]  
\mu(1, u, \chi) =& \frac{1}{2p}\Big(2(p^2 + 1) + \Tr_{\QQ(\zeta^2)/\QQ}(2) \\ 
&+ \varepsilon_C(u)\Tr_{\QQ(\zeta)/\QQ}(2) + \varepsilon_{\cc{2a}}(u)\Tr_{\QQ(\zeta)/\QQ}(-4)  \Big) \\ 
=& \frac{1}{p}(p^2 + p + 3(p - 1)\varepsilon_C(u)).
\end{align*} 
This contradicts once again $\varepsilon_C(u) \equiv 0 \mod p$ and $\varepsilon_C(u) \equiv 1 \mod q$ and the non-negativity of the multiplicities.

So finally let $G = \operatorname{P \Sigma L}(2,p^2)$ or $G = M(p^2)$ and let $u$ be of order $2q$ where $q \mid \frac{p^2 + 1}{2}$ is an odd prime. Denote by $\cc{2a}$ again the conjugacy class of involutions in $G$ which lie in $S$ and by $\cc{2b}$ and $\cc{2c}$ the other conjugacy classes of involutions in $G$, if $G = \operatorname{P \Sigma L}(2,p^2)$. If $G = M(p^2)$, the group has a unique class of involutions and the classes $\cc{2b}$ and $\cc{2c}$ can again be ignored in the following computations (see the discussion following \cref{normal_subgroups_AutPSL2p2}). Let $C_1,...,C_k$ be the conjugacy classes of elements of order $q$ in $G$ such that for some fixed elements $g$ of order $q$ we have $g^i \in C_i$. Denote by $\psi$ the character induced from the $3$-dimensional $p$-modular character $\varphi$ of $S = \PSL(2, p^2)$ described in \cref{Gruppentheorie}\eqref{PSL_characters} and by $\tau$ the character containing $S$ in the kernel and sending elements outside of $S$ to $-1$. Since 
\[\tau(u^q) = \varepsilon_{\cc{2a}}(u^q) - \varepsilon_{\cc{2b}}(u^q) - \varepsilon_{\cc{2c}}(u^q) \in \{\pm 1\} \]
and $u^q$ is normalized, we get $\varepsilon_{\cc{2a}}(u^q) \in \{0,1\}$. Now $\psi(\cc{2a}) = -2$, $\psi(\cc{2b}) = \psi(\cc{2c}) = 0$ and $\psi(C_i) = 2 + \zeta^i + \zeta^{-i} + \zeta^{ip} + \zeta^{-ip}$ for some fixed $q$-th root of unity $\zeta$. Consider $B = \{\zeta^i + \zeta^{-i} + \zeta^{ip} + \zeta^{-ip} | (i,q) = 1 \}$ as a basis of $\ZZ[\zeta + \zeta^{-1} + \zeta^{p} + \zeta^{-p}]$ and let $D$ be a representation realizing $\psi$. Separating the cases $\varepsilon_{2a}(u^q) = 1$ and $\varepsilon_{\cc{2a}}(u^q) = 0$ we have
\[D(u^q) \sim (1,1,-1,-1,-1,-1) \ \ \ \text{and} \ \ \  D(u^q) \sim (1,1,1,-1,-1,-1) \]
respectively.  In any case $u^2$ is rationally conjugate to an element of $G$, which we can prove as in the case of $G = \operatorname{P \Gamma L}(2,p^2)$ and this implies the existence of some $i$ coprime to $q$ such that
\[D(u^2) \sim (1,1,\zeta^i, \zeta^{-i}, \zeta^{ip}, \zeta^{-ip}).\]
Since $\psi(u)$ is in the $\ZZ$-span of $B$ we get
\[D(u) \sim (1,1,-\zeta^i,-\zeta^{-i}, -\zeta^{ip},-\zeta^{-ip}).\]
Thus the coefficient sum of $\psi(u)$ with respect to $B$ is $2\cdot (-\frac{q - 1}{4}) - 1 = -\frac{q + 1}{2}$. Moreover from the eigenvalues of $D(u)$ we deduce that $\varepsilon_{\cc{2a}}(u^q) = 1.$  Then $\tau(u^q) = 1$ and this yields
\[1 = \tau(u) = \varepsilon_{\cc{2a}}(u) + \tilde{\varepsilon}_{q}(u) -\varepsilon_{\cc{2b}}(u) -\varepsilon_{\cc{2c}}(u). \]
Since $u$ is normalized, we obtain $-\varepsilon_{\cc{2b}}(u) - \varepsilon_{\cc{2c}}(u) = 0$ and thus $\varepsilon_{\cc{2a}}(u) = 1 - \tilde{\varepsilon}_r(u)$.

On the other hand 
\[\psi(u) = -2\varepsilon_{\cc{2a}}(u) + \sum_{i} \varepsilon_{C_i} (2 + \zeta^i + \zeta^{-i} + \zeta^{ip} + \zeta^{-ip}). \]
Using $\varepsilon_{2a}(u) = 1 - \tilde{\varepsilon}_q(u)$, this value expressed in the basis $B$ has coefficient sum 
\[\frac{q - 1}{2}\varepsilon_{\cc{2a}}(u) - \left(\frac{q - 3}{2}\right)\tilde{\varepsilon}_q(u) 
%= \frac{q - 1}{2}(1 - \tilde{\varepsilon}_q(u)) - (\frac{q - 3}{2})\tilde{\varepsilon}_q(u) 
= \frac{q - 1}{2} +(2 - q)\tilde{\varepsilon}_q(u). \]
Comparing this with the coefficient sum computed above this means $\tilde{\varepsilon}_q(u) = -\frac{q}{2 - q}$, implying $q = 3$. However $q$ is a divisor of $\frac{p^2 + 1}{2}$ by assumption and any divisor $d$ of $\frac{p^2 + 1}{2}$ satisfies $d \equiv 1 \mod 4$, hence also $q = 3$ is impossible.
\end{proof}

\begin{remark} For $G = \PSL(2, p^3)$, $p \geq 3$ it is not known whether the integral group ring of $G$ contains a normalized unit of order $2p$.\end{remark}

\section{Some particular 4-primary groups}

For some almost simple 4-primary groups it is not possible to obtain the maximal information on the Prime Graph Question the HeLP method can give using only the package \cite{HeLPPackage}. The reason is mostly that some character and Brauer tables are known but not yet available in the Character table library of \textsf{GAP} \cite{CTblLib} or the \textsf{GAP} Atlas of Group Representations \cite{AtlasRep} or that the package \cite{HeLPPackage} cannot solve the underlying inequalities. These (series of) groups are handled in the following lemmas.  

\begin{lemma}\label{Serien} Let $G$ be a $4$-primary group isomorphic to some $\PSL(2,2^f)$, $\PSL(2,3^f)$ or $\PGL(2,3^f)$. Let $p, q \in \pi(G)$ be different primes. Then there is an element of order $pq$ in $\V(\ZZ G)$ if and only if there is an element of that order in $G$ except possibly for the following cases:
\begin{itemize}
 \item $G \simeq \PGL(2,81)$ and $p \cdot q = 3 \cdot 5$.
 \item $G \simeq \PGL(2,243)$ and $p \cdot q = 3 \cdot 11$. 
 \item $G \simeq \PSL(2,2^f)$ or $G \simeq \PSL(2, 3^f)$ or $G \simeq \PGL(2,3^f)$ and $p \cdot q = 2 \cdot 3$.
\end{itemize}
\end{lemma}

\begin{proof}
For $\PSL(2,16), \PGL(2,27)$ and $\PGL(2,81)$ the result follows from \cref{HeLPTable}.

So assume $G = \PSL(2,2^f)$ or $G = \PGL(2,3^f)$ with $f \geq 5$. Let $D$ be a representation affording the Steinberg character of $G$, cf.\ \cref{Gruppentheorie}\eqref{PSL_characters}. For a positive integer $n$, let $\zeta_n$ denote a primitive $n$-th root of unity.

 First assume that $G = \PSL(2,2^f)$. By \cref{4primaryProposition} the prime divisors of $|G|$ are $2$, $3$, $r = \frac{2^f+1}{3}$ and $s = 2^f-1$. By \cref{HertweckPSL} it suffices to consider elements of order $2r$ and $2s$. So let $u \in \V(\ZZ G)$ be a unit of order $2r$ or $2s$. Then, 
 \begin{align*}
  D(u^r) &\sim \left( \eigbox{2^{f-1}}{1}, \eigbox{2^{f-1}}{-1} \right), \\
  D(u^2) &\sim \left( 1, 1, \eigbox{3}{\zeta_r,\zeta_r^2,...,\zeta_r^{r-1}} \right)
 \end{align*}
 and
 \begin{align*}
 D(u^s) &\sim \left( \eigbox{2^{f-1}}{1}, \eigbox{2^{f-1}}{-1} \right), \\
 D(u^2) &\sim \left( 1, 1, \zeta_s,\zeta_s^2,...,\zeta_s^{s-1} \right)
\end{align*}
respectively. In any case the trace of $D(u)$ can not be integral, contradicting the fact that the Steinberg character only takes integral values.

Assume now that $G = \PGL(2,3^f)$. By \cref{4primaryProposition} either $f =5$ or the prime divisors of $|G|$ are $2$, $3$, $r = \frac{3^f + 1}{4}$ and $s = \frac{3^f - 1}{2}$. Note that in case $f = 5$ the number $\frac{3^f + 1}{4}$ is also prime, while we do not need to consider elements of order $3 \cdot 11$. Then by \cref{PGL2} it suffices in any case to consider elements of order $3r$ and $3s$. Suppose that $u \in \V(\ZZ G)$ has order $3r$ or $3s$. Then
\begin{align*}
D(u^r) &\sim \left(\eigbox{3^{f-1}}{1}, \eigbox{3^{f-1}}{\zeta_3,\zeta_3^2}\right), \\
D(u^3) &\sim \left(1, 1, 1, \eigbox{4}{\zeta_r,\zeta_r^2,...,\zeta_r^{r-1}}\right)
\end{align*}
and
\begin{align*}
D(u^s) &\sim \left(\eigbox{3^{f-1}}{1}, \eigbox{3^{f-1}}{\zeta_3,\zeta_3^2}\right), \\
D(u^3) &\sim \left(1, 1, 1, \eigbox{2}{\zeta_s,\zeta_s^2,...,\zeta_s^{s-1}}\right)
\end{align*}
respectively. Again, the sum of the eigenvalues of $D(u)$ can not be integral.
\end{proof}

\begin{remark}\label{NegativPSL2} There are $\frac{2^{f-2} + 1}{3}$ possible partial augmentations for normalized torsion units of order $6$ in $\V(\ZZ \PSL(2,2^f))$ and $f \geq 5$, and $3^{f-2} + 1$ possible partial augmentations for normalized torsion units of order $6$ in $\V(\ZZ \PGL(2,3^f))$ and $f \geq 5$. 

Assume $G = \PSL(2,2^f)$, $f \geq 5$ and $u \in \V(\ZZ G)$ is of order $6$. Note that $f$ is odd by \cref{4primaryProposition}. The maximal information we can obtain on the partial augmentations of $u$ using the HeLP method is provided by the characters given in \cref{PSL22f} since any other character of $G$ on the classes $\cc{1a}$, $\cc{2a}$ and $\cc{3a}$ can be written as a linear combination with non-negative integers of the trivial character and the given ones, cf.\ the character table of $G$ in \cite[II]{Jordan}. Moreover any irreducible Brauer character in characteristic $p \neq 2$ is liftable by \cref{Gruppentheorie}\eqref{PSL_characters}.

\begin{table}[ht] \caption{Part of the character table of $\PSL(2,2^f)$.}\label{PSL22f}
\begin{tabular}{lccc}\toprule
{ } & $\cc{1a}$ & $\cc{2a}$ & $\cc{3a}$ \\ \midrule
$\chi$ & $2^{f} - 1$ & $-1$ & $1$  \\
$\theta$ & $2^{f} - 1$ & $-1$ & $-2$  \\
\bottomrule
\end{tabular} 
\end{table}
Let $D$ be a representation of $G$ corresponding to the character $\chi$. Denote by $\zeta$ a primitive 3rd root of unity. Then
\[D(u^3) \sim \left(\eigbox{2^{f-1} - 1}{1}, \eigbox{2^{f-1}}{-1}\right) \ \ {\text{and}} \ \ D(u^2) \sim \left(\eigbox{\frac{2^f + 1}{3}}{1}, \eigbox{\frac{2^f - 2}{3}}{\zeta, \zeta^2}  \right). \]
Since the number of $-1$ in $D(u^3)$ is smaller than the number of $\zeta$ and $\zeta^2$ in $D(u^2)$, $\chi(u)$ attains its maximal value when
\[D(u) \sim \left(\eigbox{2^{f-2}}{-\zeta, -\zeta^2}, \eigbox{\left(\frac{2^f - 2}{3} - 2^{f-2}\right)}{\zeta, \zeta^2}, \eigbox{\frac{2^f + 1}{3}}{1}  \right). \]
Thus using $\varepsilon_{\cc{2a}}(u) + \varepsilon_{\cc{3a}}(u) = 1$ we obtain
\[-\varepsilon_{\cc{2a}}(u) + \varepsilon_{\cc{3a}} (u) = -1 + 2\varepsilon_{\cc{3a}}(u) = \chi(u) \leq 2^{f-2} - \left(\frac{2^f - 2}{3} - 2^{f-2}\right) + \frac{2^f + 1}{3} = 2^{f - 1} + 1  \]
and hence $\varepsilon_{\cc{3a}}(u) \leq 2^{f-2} + 1$.\\
To minimize the value of $\chi(u)$ we have
\[D(u) \sim \left(1, \eigbox{\frac{2^f - 2}{3}}{-1}, \eigbox{ \frac{2^{f - 2} + 1}{3} }{-\zeta, -\zeta^2}, \eigbox{ (2^{f - 2} - 1)}{\zeta, \zeta^2}  \right). \]
Thus
\[ -1 + 2\varepsilon_{\cc{3a}}(u) = \chi(u) \geq 1 - \frac{2^f - 2}{3} +  \frac{2^{f - 2} + 1}{3} - \left(2^{f - 2} - 1\right) = -2^{f - 1} + 3  \]
and hence $\varepsilon_{\cc{3a}}(u) \geq -2^{f-2} +2$. Since $\varepsilon_{\cc{3a}}(u) \equiv 3 \mod 6$ by \cref{Wagner} this may be changed to $\varepsilon_{\cc{3a}}(u) \geq -2^{f-2} + 5$. In between $-2^{f - 2} + 5$ and  $2^{f - 2} + 1$ there are exactly $\frac{2^{f-2} + 1}{3}$ integers being congruent $3$ modulo $6$ and these are all legitimate possibilities for $\varepsilon_{\cc{3a}}(u)$. Doing the same calculations with the character $\theta$ does not give stronger bounds for $\varepsilon_{\cc{3a}}(u)$.\\
Now let $G = \PGL(2,3^f)$ with $f$ odd and let $u \in \V( \ZZ G)$ of order 6. Denote by $\cc{3a}$ the conjugacy class of elements of order $3$ in $G$ and by $\cc{2a}$ and $\cc{2b}$ the conjugacy classes of involutions such that $\cc{2a}$ is the class in $\PSL(2,3^f)$. By the character table of $G$ given in \cite[III]{Jordan} the maximal information for the HeLP method can be extracted from the ordinary characters given in \cref{PGL23f}, since any irreducible Brauer character in characteristic $p \neq 3$ is liftable by \cref{Gruppentheorie}\eqref{PSL_characters}.

\begin{table}[ht] \caption{Part of the character table of $\PGL(2,3^f)$.}\label{PGL23f}
\begin{tabular}{lcccc}\toprule
{ } & $\cc{1a}$ & $\cc{2a}$ & $\cc{3a}$ & $\cc{2b}$ \\ \hline \\[-2ex]
$\tau$ & $1$ & $1$ & $1$ & $-1$ \\ 
$\chi$ & $3^{f} - 1$ & $2$ & $-1$ & $0$  \\
$\theta$ & $3^{f} - 1$ & $-2$ & $-1$ & $0$  \\
\bottomrule
\end{tabular} 
\end{table}
Using $\tau$ and \cref{Lemmaxy} any involution in $\V( \ZZ G)$ is rationally conjugate to an element of $G$. Assume first that $u^3 \sim \cc{2b}$. Then using $\tau$ and the fact that $u$ is normalized we get $\varepsilon_{\cc{2b}}(u) = 1$ and $\varepsilon_{\cc{2a}}(u) = -\varepsilon_{\cc{3a}}(u)$. Arguing as in the case of $\PSL(2,2^f)$ we get
\[ 2\varepsilon_{\cc{2a}}(u)  - \varepsilon_{\cc{3a}}(u) = -3\varepsilon_{\cc{3a}}(u) = \chi(u) \leq \frac{3^f-9}{2} \] 
and thus $\varepsilon_{\cc{3a}}(u) \geq -\frac{3^{f -1}  - 3}{2}$. On the other hand
\[ 2\varepsilon_{\cc{2a}}(u)  - \varepsilon_{\cc{3a}}(u) = -3\varepsilon_{\cc{3a}}(u) = \chi(u) \geq -\frac{3^f + 9}{2}, \] 
hence $\varepsilon_{\cc{3a}}(u) \leq \frac{3^{f - 1} - 3}{2}$. Using \cref{Wagner}, there are $\frac{3^{f - 2} + 1}{2}$ possibilities for $\varepsilon_{\cc{3a}}(u)$ and this fixes also $\varepsilon_{\cc{2a}}(u)$. The bounds that can be obtained using $\theta$ do not give stronger restrictions on the partial augmentations of $u$.

Now assume $u^3$ is rationally conjugate to elements in $\cc{2a}$. Then using $\tau$ we have $\varepsilon_{\cc{2b}}(u) = 0$ and hence $\varepsilon_{\cc{2a}}(u) = 1 - \varepsilon_{\cc{3a}}(u)$. Via $\chi$ we get
\[ 2\varepsilon_{\cc{2a}}(u) - \varepsilon_{\cc{3a}}(u) = 2 - 3\varepsilon_{\cc{3a}}(u) = \chi(u) \leq \frac{3^f - 5}{2} \] 
and 
\[ 2\varepsilon_{\cc{2a}}(u) - \varepsilon_{\cc{3a}}(u) = 2 - 3\varepsilon_{\cc{3a}}(u) = \chi(u) \geq -\frac{3^f -1}{2} \] 
implying, using \cref{Wagner}, that $-\frac{3^{f - 1} - 3}{2} \leq \varepsilon_{\cc{3a}}(u) \leq \frac{3^{f - 1} -3}{2}$. This gives another $\frac{3^{f - 2} + 1}{2}$ possible partial augmentations for $u$.

Let $G = \PSL(2, 3^f)$ for odd $f = 2e + 1 \geq 3$. As the number of admissible possibilities of partial augmentations for normalized units of order $6$ in $\V(\ZZ G)$ depends on the number of possibilities for units of order $3$ we will show that the HeLP method does not suffice to show that there is no unit of order $6$ in $\V(\ZZ G)$ by verifying that a unit $u \in \V(\ZZ G)$ with the partial augmentations \[(\varepsilon_{\cc{2a}}(u^3), \varepsilon_{\cc{3a}}(u^2), \varepsilon_{\cc{3b}}(u^2), \varepsilon_{\cc{2a}}(u), \varepsilon_{\cc{3a}}(u), \varepsilon_{\cc{3b}}(u)) = (1, 1, 0, -2, 2, 1)\] can not be ruled out. As such a unit would have admissible partial augmentations in $\ZZ \PGL(2, 3^f)$, new obstructions can only arise from irreducible characters of $G$ that stay irreducible after being induced to $\PGL(2, 3^f)$. These characters are listed in \cref{PSL23f}.

\begin{table}[ht] \caption{Part of the character table of $\PSL(2,3^f)$, $f = 2e + 1$, $\zeta$ a primitive $3$rd root of unity.}\label{PSL23f}
\begin{tabular}{lcccc}\toprule
{ } & $\cc{1a}$ & $\cc{2a}$ & $\cc{3a}$ & $\cc{3b}$ \\ \hline \\[-2ex]
$\eta$ & $\frac{3^{f} - 1}{2}$ & $1$ & $\frac{3^e - 1}{2} + 3^e\zeta$ & $\frac{3^e - 1}{2} + 3^e\zeta^2$  \\
$\eta'$ & $\frac{3^{f} - 1}{2}$ & $1$ & $\frac{3^e - 1}{2} + 3^e\zeta^2$ & $\frac{3^e - 1}{2} + 3^e\zeta$  \\
\bottomrule
\end{tabular} 
\end{table}

Let $D$ be a representation affording $\eta$ and let $\zeta$ be a fixed primitive $3$rd root of unity. Assume that $D(u)$ has the following eigenvalues for some normalized unit $u$
\begin{align*} D(u) \sim  \left(\eigbox{\frac{3^{2e} - 5}{4}}{1}, \eigbox{\frac{3^{2e} + 3}{4}}{-1}, \eigbox{\frac{3^{2e} + 3}{4}}{\zeta, \zeta^2}, \eigbox{\frac{3^{2e} - 2\cdot 3^e - 3}{4}}{-\zeta}, \eigbox{\frac{3^{2e} + 2\cdot 3^e - 3}{4}}{-\zeta^2} \right). \end{align*} Then this $u$ fulfills $\eta(u) = \frac{3^e - 7}{2} + 3^e\zeta = -2\eta(\cc{2a}) + 2\eta(\cc{3a}) + \eta(\cc{3b})$ and \begin{align*} D(u^3) \sim D(\cc{2a}) & \sim  \left(\eigbox{\frac{3^f + 1}{4}}{1}, \eigbox{\frac{3^f - 3}{4}}{-1} \right), \\ D(u^2) \sim D(\cc{3a}) & \sim  \left(\eigbox{\frac{3^{2e}-1}{2}}{1}, \eigbox{\frac{3^{e}(3^{e} + 1)}{2}}{\zeta}, \eigbox{\frac{3^{e}(3^{e} - 1)}{2}}{\zeta^2} \right). \end{align*} So all obstructions that can be derived from the HeLP method for this character are fulfilled. We can take the images under the non-trivial element of $\Gal(\QQ(\zeta)/\QQ)$ as the eigenvalues of the image of $u$ under a representation affording $\eta'$, so this character can also not eliminate this possibility.
\end{remark}

\begin{lemma}\label{PSL281}
Let $G$ be an almost simple group with socle $S = \PSL(2,81)$. Then units of order $5$ in $\V( \ZZ G)$ are rationally conjugate to elements of $G$. If $u \in \V(\ZZ G)$ is of order $15$, then the restrictions on the partial augmentations of $u$ using the HeLP method for Brauer tables of $G$ are the same as when using the ordinary character table of $G$. 
\end{lemma}

\begin{proof}
Let the outer automorphism group of $S$ be generated by $g$ and $\sigma$, where $\sigma$ is the automorphism induced by the Frobenius automorphism of $\mathbb{F}_{81}$ and $g$ is an involution in $\PGL(2,81)$ outside of $S$, see \cref{Gruppentheorie}\eqref{PSL_autos}. There are two conjugacy classes of elements of order $5$ in $S$ and these fuse into one class in $G$, if and only if $G/S$ contains elements of order $4$. If there are two classes of elements of order $5$ using the $3$-modular Brauer character called $\varphi$ in \cref{Gruppentheorie}\eqref{PSL_characters}, if defined for $G$, or the character induced by $\varphi$ and performing the same computations as in the proof of \cref{PGL2}, we obtain that units of order $5$ in $\V(\ZZ G)$ are rationally conjugate to elements of $G$.

All $2$-modular Brauer tables are available in \cite{CTblLib}, so we have to show that no new information can be deduced for the partial augmentations of units of $15$ using the $41$-modular Brauer characters of $G$ compared to the ordinary character table. If $G/S$ is elementary abelian, all Brauer tables of $G$ in characteristic $41$ are available in \cite{CTblLib} and the result follows. So let $G/S$ contain elements of order $4$. Let $k$ be a field of characteristic $41$ large enough to admit all $41$-modular representation of $G$. Let $M$ be a simple $kG$-module, then $M$ is also a $kS$-module. Looking at the Brauer table of $S$ and comparing it with the ordinary table of $G$, we obtain that the character associated to $M$ could only give more information, if $M$ would be $41$-dimensional. If $G/S$ is generated by $\sigma$, then $G$ has a $41$-dimensional simple module, since there is even a $41$-dimensional ordinary representation of $G$ and in this case the ordinary characters suffice to show that there are no elements of order $15$ in $\V(\ZZ G)$, cf.\ \cref{HeLPTable}. However $\PGL(2,81)$ does not possess a simple $41$-dimensional $41$-modular module and thus $G$ does not, if $G/S$ is not generated by $\sigma$. 
\end{proof}

For certain groups where not all characters we need are directly accesible in \textsf{GAP} we give expanded arguments.
\begin{lemma}\label{PSL2243}
Let $G = \PSL(2,243)$. There are no units of order $33$ in $\V(\ZZ G)$.
\end{lemma}

\begin{proof}
The character table of $G$ is not contained in \cite{CTblLib}. $G$ contains exactly five conjugacy classes of elements of order $11$ and two conjugacy classes of elements of order $3$. By \cite{Schur} $G$ possesses a character as given in \cref{PSL2243Char}.

\begin{table}[ht] \caption{Part of the character table of $\PSL(2,243)$.}\label{PSL2243Char}
\begin{tabular}{lcccccccc}\toprule
{ } & $\cc{1a}$ & $\cc{3a}$ & $\cc{3b}$ & $\cc{11a}$ & $\cc{11b}$ & $\cc{11c}$ & $\cc{11d}$ & $\cc{11e}$ \\ \midrule
$\chi$ & $121$ & $\frac{-1+\sqrt{-243}}{2}$ & $\frac{-1-\sqrt{-243}}{2}$ & $0$ & $0$ & $0$ & $0$ & $0$  \\
\bottomrule
\end{tabular} 
\end{table}

This character may be used to show that there are no units of order $33$ in $\V(\ZZ G)$
via the package \cite{HeLPPackage} and the following code.

\begin{verbatim}
gap> H := PSL(2,243);; C := CharacterTable(H);;
gap> h := Size(ConjugacyClasses(C));
124
gap> o3 := Positions(OrdersClassRepresentatives(C), 3);
[ 2, 3 ]
gap> chi := ListWithIdenticalEntries(h, 0);;
gap> chi[1] := 121;;
gap> chi[o3[1]] := (-1 + Sqrt(-243))/2;; chi[o3[2]] := (-1 - Sqrt(-243))/2;;
gap> chi := ClassFunction(C, chi);;
gap> HeLP_WithGivenOrderSConstant([chi], 11, 3);;
#I  Number of solutions for elements of order 33: 0; stored in HeLP_sol[33].
\end{verbatim}
\end{proof}

\begin{remark}\label{NegativPGL2243}
Let $G = \PGL(2, 243)$. We will show that the HeLP method is not sufficient to prove that there are no units if order $33$ in $\V(\ZZ G)$.

There are five conjugacy classes of elements of order $11$ in $G$ and one class of elements of order $3$. Looking at the character table of $G$ given in \cite{Schur} we see that the maximal possible information using the ordinary character table and the HeLP method can be obtained from the characters given in \cref{PGL2243}.

\begin{table}[ht] \caption{Part of the character table of $\PGL(2,243)$, where $z$ denotes some fixed $11$-th root of unity.}\label{PGL2243}
\begin{tabular}{lccccccc}\toprule
{ } & $\cc{1a}$ & $\cc{3a}$ & $\cc{11a}$ & $\cc{11b}$ & $\cc{11c}$ & $\cc{11d}$ & $\cc{11e}$ \\ \midrule
$\chi$ & $242$ & $-1$ & $0$ & $0$ & $0$ & $0$ & $0$  \\
$\psi_1$ & $244$ & $1$ & $z + z^{-1}$ & $z^2 + z^{-2}$ & $z^4 + z^{-4}$ & $z^3 + z^{-3}$ & $z^5 + z^{-5}$  \\
$\psi_2$ & $244$ & $1$ & $z^2 + z^{-2}$ & $z^4 + z^{-4}$ & $z^3 + z^{-3}$ & $z^5 + z^{-5}$ & $z + z^{-1}$  \\
$\psi_3$ & $244$ & $1$ & $z^4 + z^{-4}$ & $z^3 + z^{-3}$ & $z^5 + z^{-5}$ & $z + z^{-1}$ & $z^2 + z^{-2}$  \\
$\psi_4$ & $244$ & $1$ & $z^3 + z^{-3}$ & $z^5 + z^{-5}$ & $z + z^{-1}$ & $z^2 + z^{-2}$ & $z^4 + z^{-4}$  \\
$\psi_5$ & $244$ & $1$ & $z^5 + z^{-5}$ & $z + z^{-1}$ & $z^2 + z^{-2}$ & $z^4 + z^{-4}$ & $z^3 + z^{-3}$  \\
\bottomrule
\end{tabular} 
\end{table}

Assuming that $u^3$ is rationally conjugate to elements in $\cc{11a}$ the partial augmentations $\varepsilon_{\cc{3a}}(u) = 12$, $\varepsilon_{\cc{11c}}(u) = -11$ and $\varepsilon_C(u) = 0$ for any other class $C$ of elements of order $11$ satisfy all HeLP constraints imposed by the characters in \cref{PGL2243}.
This can be seen with the HeLP package \cite{HeLPPackage} via the following \textsf{GAP} code:  

{\footnotesize \begin{verbatim}
gap> G := PGL(2,243);; C := CharacterTable(G);;
gap> h := Size(ConjugacyClasses(C));
245
gap> o := OrdersClassRepresentatives(C);; o3 := Positions(o, 3); o11 := Positions(o, 11);
[ 245 ]
[ 2, 3, 4, 5, 6 ]
gap> chi := ListWithIdenticalEntries(h, 0);;
gap> chi[1] := 242;;
gap> chi[o3[1]] := -1;;
gap> chi := ClassFunction(C, chi);;
gap> y := E(11) + E(11)^-1;; K := Field(y);;
gap> S := GaloisGroup(K);; s := GeneratorsOfGroup(S)[1];;
gap> psi := [];;     # list of the characters psi_1, ..., psi_5
gap> for j in [1..5] do psi[j]:=ListWithIdenticalEntries(h, 0); psi[j][1]:=244; psi[j][o3[1]]:=1;
> for k in [1..Size(o11)] do psi[j][o11[k]] := y^(s^(j+k)); od;
> psi[j] := ClassFunction(C, psi[j]); od;

gap> HeLP_sol[33] := [ [ [1], [1, 0, 0, 0, 0], [12, 0, 0, -11, 0, 0] ] ];
[ [ [ 1 ], [ 1, 0, 0, 0, 0 ], [ 12, 0, 0, -11, 0, 0 ] ] ]
gap> HeLP_VerifySolution(Concatenation([chi], psi), 33);
[ [ [ 1 ], [ 1, 0, 0, 0, 0 ], [ 12, 0, 0, -11, 0, 0 ] ] ]
\end{verbatim}
}
%{\footnotesize \verbatiminput{PGL_2_243.log}}
Since every $2$-Brauer character and $61$-Brauer character is liftable by \cref{Gruppentheorie}\eqref{PSL_characters}, the ordinary table of $G$ provides the maximal information, and this finishes the proof. 
The number of possible partial augmentations for elements of order $33$ given in \cref{HeLPTable} was computed using the program Normaliz \cite{Normaliz}.
\end{remark}

\begin{lemma}\label{PSL317} Let $G = \Aut(\PSL(3,17))$ or $G = \PSL(3,17)$.  Let $p, q \in \pi(G)$ be different primes. Then there is an element of order $pq$ in $\V(\ZZ G)$ if and only if there is an element of that order in $G$ except possibly for $p \cdot q = 3\cdot 17$.
\end{lemma}

\begin{proof} Using \cite{Magma} or \cite{SimpsonFrame} we can calculate the characters of $G = \Aut(\PSL(3,17))$ given in \cref{chartab_L_3_17_2}.
\begin{table}[ht] \caption{Parts of the character table of $\Aut(\PSL(3, 17))$. The class $\cc{307x}$ stands for any class of elements of order $307$, dots indicate zeros, blank spots are entries not needed.}\label{chartab_L_3_17_2}
\begin{tabular}{cccccccc}\toprule & \cc{1a} & \cc{2a} & \cc{2b} & \cc{3a} & \cc{17a} & \cc{17b} & \cc{307x} \\ \midrule
$\chi_{1}$ & $1$ & $1$ & $-1$ & $1$ & $1$ & $1$ & $1$  \\ % X.2
$\chi_{306}$ & $306$ & $18$ & $.$ & $.$ & $17$ & $.$ & $-1$  \\ % X.3
$\chi_{4912}$ & $4912$ & $16$ & $16$ & $1$ & $-1$ & $-1$ & $.$ \\ % X.18
$\chi_{9216}$ & $9216$ & $.$ & $.$ & $.$ & $-32$ & $2$ &  \\ % X.48
\bottomrule
\end{tabular} 
\end{table}
Applying \cref{Lemmaxy} to the non-trivial linear character of $G/\PSL(3,17) \simeq C_2$ (corresponding to $\chi_{\cc{1}}$) we obtain that involutions in $\V(\ZZ G)$ are rationally conjugate to elements in $G$.
Let $u \in \V(\ZZ G)$. Assuming $u$ has order $2 \cdot 307$, $3 \cdot 307$ or $17 \cdot 307$ analogues calculations as in the proof of \cref{Serien} prove the non-existence of $u$.

Assume that $o(u) = 2\cdot 307$. Recall that $\tilde{\varepsilon}_{307}(u)$ denotes the sum of all partial augmentations of $u$ at classes of elements of order $307$. Using that $u$ is normalized, we have \[ \chi_{1}(u) = \varepsilon_{\cc{2a}}(u) - \varepsilon_{\cc{2b}}(u) +\tilde{\varepsilon}_{307}(u) = 1 - 2\varepsilon_{\cc{2b}}(u) \in \{ \pm 1\} \] and thus $\varepsilon_{\cc{2b}}(u) \in \{0, 1\}$. We have $\chi_{4912}(u^2) = 0$ and $\chi_{4912}(u^{307}) = 16$. Denote by $D_j$ a representation having character $\chi_j$. Then
\begin{align*}
D_{4912}(u^2) \sim \left(\eigbox{16}{1}, \eigbox{16}{\zeta, ..., \zeta^{306}} \right), \\
D_{4912}(u^{307}) \sim \left(\eigbox{2464}{1}, \eigbox{2448}{-1} \right)
\end{align*}
for a primitive $307$-th root of unity $\zeta$. Note that $\chi_{4912}(u) = 16(\varepsilon_{\cc{2a}}(u) + \varepsilon_{\cc{2b}}(u)) \in \ZZ$. Hence $\varepsilon_{\cc{2a}}(u) + \varepsilon_{\cc{2b}}(u) = 1$, implying $(\varepsilon_{\cc{2a}}(u), \varepsilon_{\cc{2b}}(u),\tilde{\varepsilon}_{307}(u)) \in \{(1, 0, 0), (0, 1, 0)\}$ contradicting \cref{Wagner}.

Now assume that $o(u) = 3\cdot 307$. From the character values we get
\begin{align*}
D_{306}(u^3) \sim \left(\eigbox{1}{\zeta, ..., \zeta^{306}} \right), \\
D_{306}(u^{307}) \sim \left(\eigbox{102}{1}, \eigbox{102}{\xi, \xi^2} \right)
\end{align*}
where $\zeta$ denotes a again a primitive $307$-th root of unity and $\xi$ a primitive third root of unity. As $\chi_{306}(u) = -\tilde{\varepsilon}_{307}(u) \in \ZZ$ this yields a contradiction.

Now assume that $o(u) = 17\cdot 307$. Then by considering the character values we obtain
\begin{align*}
D_{4912}(u^{17}) \sim \left(\eigbox{16}{1}, \eigbox{16}{\zeta, ..., \zeta^{306}} \right), \\
D_{4912}(u^{307}) \sim \left(\eigbox{288}{1}, \eigbox{289}{\eta, ..., \eta^{16}} \right)
\end{align*}
for a primitive $307$-th root of unity $\zeta$ and a primitive $17$-th root of unity $\eta$. Now as $\chi_{4912}(u) \in \ZZ$, the $288$ eigenvalues $1$ of $D_{4912}(u^{307})$ have to either all be multiplied with $1$'s or a whole block of $307$-ths roots of unity. But as $288 < 306$ and the multiplicity of the eigenvalue $1$ in $D_{4912}(u^{17})$ is $16$, this is impossible.
\end{proof}

\begin{remark}\label{NegativPSL317} We will show that for $G = \PSL(3, 17)$ or $G = \Aut(\PSL(3,17))$ the constraints imposed by the HeLP method allow exactly 126 possible partial augmentations for elements of order $3 \cdot 17 = 51$.

So let $u \in \V(\ZZ G)$ and assume that $o(u) = 51$. Then using the HeLP package in the same way as in the proof of \cref{PSL2243} 
with the characters $\chi_{306}$, $\chi_{4912}$ and $\chi_{9216}$ from \cref{chartab_L_3_17_2} one obtains that there are exactly 126 tuples of partial augmentations admissible according to the HeLP constraints for these three characters. 

We need to show that this is the maximal possible information available using also the modular characters. Note that we do not have to consider $17$-modular characters. The decomposition matrices of $\PSL(3,17)$ are, in principle, given in \cite{James}, but to get them explicitly requires some computational effort, so we give other arguments. 

Let $G = \Aut(\PSL(3,17))$ and $H = \PSL(3,17)$. The Sylow $3$- and Sylow $307$-subgroups of $G$ and $H$ are cyclic. Thus we can investigate the $3$-modular and $307$-modular blocks of $G$ and $H$ using the theory of Brauer Tree Algebras, cf. \cite[§ 62]{CR2}. From the ordinary character table it can be determined which ordinary characters lie in the same block for a given prime \cite[§ 56]{CR2}. By the number of $3$-regular conjugacy classes in $G$ and $H$ and the degrees of ordinary characters lying in the same block, we conclude that any $3$-block contains at most two irreducible modular characters, i.e. a corresponding Brauer Tree has at most $2$ edges. Thus any $3$-Brauer character is liftable.

All $307$-blocks apart from the main block contain exactly one irreducible ordinary character. By the number of $307$-regular conjugacy classes of $G$ and $H$ we conclude that the main block of $G$ contains six, while the main block for $H$ consist of three irreducible modular characters. The irreducible $306$-dimensional ordinary representation of $H$ is a deleted permutation representation coming from the natural $2$-transitive permutation action on the projective plane over $\mathbb{F}_{17}.$ Thus by \cite[V, Beispiel 20.3]{HuppertI} the reduction of a corresponding lattice has a trivial constituent. This implies that the vertices corresponding to the trivial character $\chi_1$ and the character $\chi_{306}$ of degree $306$ are connected in the Brauer Tree of the main block of $H$. Since the trivial character always sits at an end of the tree, the tree has the form 
\[
\begin{tikzpicture}
\node[label=north:{$\chi_1$}] at (0,1.5) (1){};
\node[label=north:{$\chi_{306}$}] at (1.5,1.5) (2){};
\node[label=north:{$\chi_{4913}$}] at (3,1.5) (3){};
\node[label=north:{$\chi_{4608}$}] at (4.5,1.5) (4){};
\foreach \p in {1,2,3,4}{
\draw[fill=white] (\p) circle (.075cm);
}
\draw[fill=black] (4) circle (.075cm);
\draw (.075,1.5)--(1.425,1.5);
\draw (1.575,1.5)--(2.925,1.5);
\draw (3.075,1.5)--(4.425,1.5);
\end{tikzpicture}.
\]
%(1---306---4913---4608).
By restricting the characters of $G$ to characters of $H$ we obtain that the Brauer Tree of the main block of $G$ looks as follows: 
\[
\begin{tikzpicture}
\node[label=north:{$\chi_1$}] at (0,1.5) (1){};
\node[label=north:{$\chi_{306}$}] at (1.5,1.5) (2){};
\node[label=north:{$\chi_{4913}$}] at (3,1.5) (3){};
\node[label=north:{$\chi_{9216}$}] at (4.5,1.5) (4){};
\node[label=north:{$\chi_{4913'}$}] at (6,1.5) (5){};
\node[label=north:{$\chi_{306'}$}] at (7.5,1.5) (6){};
\node[label=north:{$\chi_{1'}$}] at (9,1.5) (7){};
\foreach \p in {1,2,3,4,5,6,7}{
\draw[fill=white] (\p) circle (.075cm);
}
\draw[fill=black] (4) circle (.075cm);
\draw (.075,1.5)--(1.425,1.5);
\draw (1.575,1.5)--(2.925,1.5);
\draw (3.075,1.5)--(4.425,1.5);
\draw (4.575,1.5)--(5.925,1.5);
\draw (6.075,1.5)--(7.425,1.5);
\draw (7.575,1.5)--(8.925,1.5);
\end{tikzpicture}.
\]
%(1--306--4913--9216--4913'--306'--1').
We conclude that $H$ possesses one irreducible $307$-modular Brauer character that can not be lifted, while $G$ possesses four. These four however fall into pairs coinciding on the conjugacy classes of elements of order $3$ and $17$. All of these characters however can not exclude any partial augmentation computed before.\

All $2$-modular characters of $H$ are liftable, but the arguments here are more elaborated. The main block is the only block possessing more than one irreducible Brauer character, namely three. Just from comparing character values we get the following decomposition matrix $D$, where repeating lines are omitted.

\[D =\begin{pmatrix}
1 & 0 & 0  \\
a_1 & a_2 & a_3 \\
1 + a_1 & a_2 & a_3 \\
b_1 & b_2 & b_3 \\
1 + b_1 & b_2 & b_3 \\
1 + a_1 + b_1 & a_2 + b_2 & a_3 + b_3 \\
2 + 2a_1 + b_1 & 2a_2 + b_2 & 2a_3 + b_3 \\
\end{pmatrix}\]

The second line of $D$ corresponds to the canceled natural permutation module on the $307$ points of the projective plane over $\mathbb{F}_{17}$. This action is $2$-transitive and a short computations implies that this is a simple $\mathbb{F}_2H$-module, a $306$-dimensional invariant space is spanned by the vectors containing an odd number of $0$'s. Thus $a_1 = a_3 = 0$ and $a_2 = 1$. The fifth line of $D$ corresponds to the Steinberg character of $H$ and from \cite[Theorem A]{HissSteinberg}, we get $b_1 = 0$. Since the decomposition matrix of $H$ can be brought to be in unitriangular form by \cite[Theorem 3.8]{Dipper} we obtain $b_3 = 1$. By dimensions of degrees we obtain then $b_2 \leq 16$. Since the determinant of the Cartan matrix $D\cdot D^t$ is a power of $2$ \cite[Theorem 18.25]{CR1} this implies $b_2 = 0$ and we are done. 

For $G$ and for the prime $2$ also only the main block contains more then one irreducible Brauer character. From the dimensions of the representations in the main block it follows that apart from a $4912$-dimensional irreducible Brauer character corresponding to the $4912$-dimensional character in the main block of $H$, any irreducible Brauer character is liftable. This character does not provide new information on the possible partial augmentations of units of order $51$.
\end{remark}

\begin{lemma}\label{S47} Let $G$ be $\PSp(4,7)$ or $\Aut(\PSp(4,7))$ and let $p, q \in \pi(G)$ different primes. Then there are normalized units of order $p\cdot q$ in $\V(\ZZ G)$ if and only if there exist elements of this order in $G$ except possibly for $p\cdot q = 5\cdot 7$. \end{lemma}

\begin{proof} For both groups the ordinary character tables and one character in characteristic $7$ are available in \cite{AtlasRep}. We will exclude the existence of normalized units of order $2\cdot 5$ in $\ZZ \PSp(4,7)$ and $3\cdot 5$ in $\ZZ \Aut(\PSp(4,7))$.

First let $G = \PSp(4,7)$ and assume that $u \in \V(\ZZ G)$ is of order $10$. Note that $G$ contains two conjugacy classes of elements of order $2$, which can be distinguished by the sizes of their centralizers, and one conjugacy class of elements of order $5$. To disprove the existence of $u$ we use an ordinary character $\chi$ of degree $175$, which is uniquely determined by its values on involutions, together with a $7$-Brauer character $\varphi$ coming from the isomorphism $\PSp(4, 7)\simeq \Omega(5, 7)$. The values of these characters are given in \cref{chartab_PSp_4_7}.

{\footnotesize
\begin{verbatim}
gap> C := CharacterTable("S4(7)");;
gap> oC := OrdersClassRepresentatives(C);;
gap> G := AtlasGroup("S4(7)", Characteristic, 7, Dimension, 5);;
gap> CC := ConjugacyClasses(G);;
gap> oG := List(CC, C -> Order(Representative(C)));;
gap> phi := ListWithIdenticalEntries(Size(Irr(C)), 0);;
gap> phi[1] := BrauerCharacterValue(Representative(CC[Position(oC, 1)]));;
gap> phi[Position(oC, 5)] := BrauerCharacterValue(Representative(CC[Position(oG, 5)]));;
gap> for j in Positions(oG, 2) do for k in Positions(oC, 2) do
>   if SizesCentralizers(C)[k] = Order(Centralizer(G, Representative(CC[j]))) then
>   phi[k] := BrauerCharacterValue(Representative(CC[j]));
> fi; od; od;
gap> phi := ClassFunction(C, phi);;

gap> HeLP_WithGivenOrder([phi], 2);;
#I  Number of solutions for elements of order 2: 3; stored in HeLP_sol[2].
gap> chi := First(Irr(C), x -> x{Positions(oC,2)} = [31, 7]);;
gap> HeLP_WithGivenOrder([ chi, phi ], 2*5);;
#I  Number of solutions for elements of order 10: 0; stored in HeLP_sol[10].
\end{verbatim}
}
% {\footnotesize \verbatiminput{S_4_7.log}}

\begin{table}[ht] \caption{Parts of some characters of $G \in \{\PSp(4, 7), \Aut(\PSp(4, 7))\}$.}\label{chartab_PSp_4_7}
\subfloat[$G = \PSp(4, 7)$]{
\begin{tabular}{ccccc}\toprule & \cc{1a} & \cc{2a} & \cc{2b} & \cc{5a} \\ \midrule
$\chi = \chi_{8/175b}$ & $175$ & $31$ & $7$ & $0$  \\ 
$\varphi$ & $5$ & $-3$ & $1$ & $0$  \\ 
\bottomrule
\end{tabular}}
\qquad
\subfloat[$G = \Aut(\PSp(4, 7))$]{
\begin{tabular}{ccccc}\toprule & \cc{1a} & \cc{3a} & \cc{3b} & \cc{5a} \\ \midrule
$\chi' = \chi_{3/50}$ & $50$ & $2$ & $8$ & $0$  \\ 
$\varphi'$ & $5$ & $2$ & $-1$ & $0$  \\ 
\bottomrule
\end{tabular}}
\end{table}

Now assume that $G = \Aut(\PSp(4,7))$ and that $u \in \V(\ZZ G)$ is of order $15$. The character table of $G$ is available in \textsf{GAP} via the command \texttt{CharacterTable("S4(7).2");}.  With a similar code as above one can exclude the existence of units of order $15$ in $\ZZ G$ using the $7$-Brauer character $\varphi'$ for elements of order $3$ and the characters $\varphi'$ and $\chi'$ for elements of order $15$. The relevant values are given in \cref{chartab_PSp_4_7}. Note that the conjugacy classes of elements of order $3$ can be distinguished by the structures of their centralizers, in particular by the number of involutions contained therein, or by power maps: the class $\cc{3b}$ contains all $16$th powers of elements of order $48$.
\end{proof}

\begin{remark}\label{NegativS47}
Let $G = \PSp(4,7)$ or $G = \Aut(\PSp(4,7))$. The HeLP method is not sufficient to decide if there exist units of order $5 \cdot 7 = 35$ in $\V(\ZZ G)$.

First let $G = \PSp(4,7)$. Since there is only one conjugacy class of elements of order $5$ in $G$, we do not need to consider $7$-modular Brauer characters. The irreducible $3$-modular and $5$-modular Brauer characters of $G$ are known. Every irreducible $3$-modular Brauer character is liftable by \cite[Theorem 4.1]{Whiteq1} and there are excatly two non-liftable irreducible $5$-modular Brauer characters by \cite[Theorem 3.9]{WhiteBrauer}. 
These can be obtained in \textsf{GAP} via the code given below and are called \texttt{eta1} and \texttt{eta2} there.
The irreducible $2$-modular Brauer characters are not completly known, to our knowledge. They are determined in \cite[Theorem 3.1]{White2} up to an parameter $x$ which for our $G$ can take values $1$, $2$ or $3$. The constraints for $x = 1$ and $x = 2$ are weaker then for $x = 3$ and it thus will suffice to consider the case $x = 3$, since we show the existence of non-trivial solutions. Then there are five non-liaftable irreducible $2$-modular Brauer characters of $G$ called $\varphi_1$ to $\varphi_5$ in \cite[Theorem 3.1]{White2}. Note that $\varphi_4$ and $\varphi_5$ are liftable as characters of the corresponding symplectic group, but not as characters of the projective symplectic group. The HeLP restrictions provided by these characters leave exactly $9$ possible non-trivial partial augmentations for elements of order $35$ in $\V(\ZZ G)$.
Thus the following code demonstrates that there are $9$ possible non-trivial partial augmentations satisfying the HeLP constraints for elements of order $35$ in $\V(\ZZ G)$.

{\footnotesize
\begin{verbatim}
gap> C := CharacterTable("S4(7)");;
gap> eta1 := Irr(C)[9] - Irr(C)[1];;
gap> eta2 := Irr(C)[37] - Irr(C)[4];; # 5-modular Brauer characters

gap> phi4 := Irr(C)[2] - Irr(C)[1];;
gap> phi5 := Irr(C)[3] - Irr(C)[1];;
gap> phi6 := Irr(C)[4];;
gap> phi3 := Irr(C)[7] - Irr(C)[1];;
gap> x := 3;;
gap> phi1 := Irr(C)[18] - (x-1)*phi6;;
gap> phi2 := Irr(C)[19] - (x-1)*phi6;; # 2-modular Brauer characters

gap> L5 := [eta1, eta2];;
gap> L2 := [phi1, phi2, phi3, phi4, phi5];;
gap> HeLP_WithGivenOrder(Concatenation(Irr(C), L2, L5), 7);;
#I  Number of solutions for elements of order 7: 4706; stored in HeLP_sol[7].
gap> HeLP_WithGivenOrder(Concatenation(Irr(C), L2), 5*7);;
#I  Number of solutions for elements of order 35: 9; stored in HeLP_sol[35].
\end{verbatim}
}
%{\footnotesize \verbatiminput{S_4_7_ord35.log}}

For $G = \Aut(\PSp(4,7))$ analogues computations demonstrate that there are $4$ possible non-trivial partial augmentations satisfying the HeLP constraints. The irreducible Brauer characters can be derived from the ordinary table and the arguments above. In particular an irreducible $2$-modular Brauer character of $\PSp(4,7)$ induced to $G$ is irreducible if and only if it is rational-valued on the conjugacy classes of elements of order $7$.
\end{remark}

\begin{example}\label{ExplicitExample}
We give an explicit example how the HeLP package can be used to check the information given in \cref{HeLPTable}.

{\footnotesize 
\begin{verbatim}
gap> LoadPackage("help");;
gap> C := CharacterTable("L2(81).(2x4)");;    
gap> HeLP_WithGivenOrder(C, 15);
#I  Number of solutions for elements of order 15: 1; stored in HeLP_sol[15].
[ [ [ 1 ], [ 1 ], [ 6, -5 ] ] ]
gap> HeLP_WithGivenOrder(Irr(C){[13, 33]}, 15);;
#I  Number of solutions for elements of order 15: 1; stored in HeLP_sol[15].
gap> HeLP_WithGivenOrderSConstant(Irr(C){[9]}, 41, 3);;
#I  Number of solutions for elements of order 123: 0; stored in HeLP_sol[123].
gap> HeLP_WithGivenOrderSConstant(Irr(C){[9]}, 41, 5);;
#I  Number of solutions for elements of order 205: 0; stored in HeLP_sol[205].
\end{verbatim}
}
%% {\footnotesize \verbatiminput{Explicit.log}}

\end{example}

\begin{example}\label{code_induced_characters}
We give an example, how induced characters can be used in the package. This is necessary for some groups. More specifically, we will prove here that for $G = \Aut(\PSU(3,8))$ there are no elements of order $7\cdot 19$ in $\V(\ZZ G)$.

{\footnotesize
\begin{verbatim}
gap> C := CharacterTable("U3(8)");;
gap> G := PSU(3,8);;
gap> A := AutomorphismGroup(G);;
gap> AllCharacterTableNames(Size,Size(A));
[ "3.U3(8).6", "3.U3(8).S3" ]
#This means: The character table of the automorphism group A of PSU(3,8) is not available in GAP. 
gap> NN := Filtered(NormalSubgroups(A), N -> Order(N) = Order(G));
[ <group of size 5515776 with 2 generators> ]
gap> H := NN[1];;      #Subgroup of A isomorphic to G              
gap> CharacterTableWithStoredGroup(H,C);;
gap> D := CharacterTable(H);; 
gap> chi := InducedClassFunction(Irr(D)[2],A);;
gap> HeLP_WithGivenOrder([chi],7*19);;
#I  Number of solutions for elements of order 133: 0; stored in HeLP_sol[133].
\end{verbatim}
}
% {\footnotesize \verbatiminput{PSU_3_8.log}}

\end{example}

\section{Proof of \cref{MainTheorem}}\label{results}

To prove \cref{MainTheorem} we give in \cref{HeLPTable} below the complete information for all almost simple $4$-primary groups in a compact form such that the results are easy to reproduce\footnote{A file to check the claimed data is available on the website of the first author: \url{http://homepages.vub.ac.be/abachle/\#research}.}. For every simple $4$-primary group $S$, it contains every $4$-primary subgroup of $\Aut(S)$ up to isomorphism. Almost simple groups having the same socle are grouped together and are separated by a single line, while groups having different socles are separated by two lines. It is to read in the following way. \\
\begin{itemize}
\item The first column contains one or more names of the group. If the ordinary character table of the group is available in the \textsf{GAP} character table library, a name of it from that library is given in quotes. 
For instance, the character table of $\Aut(\PSL(2,16))$ can then be obtained in \textsf{GAP} by \texttt{CharacterTable("L2(16).4");}.
\item The second column contains the prime graph of the group.
\item The third column records which orders of torsion units need to be checked to obtain a positive answer to the Prime Graph Question for the given group. The critical order is only given for the biggest group it has to be checked (in case the existence of units of that order can be excluded for that group).  So if $H$ is a subgroup of $G$ such that $G$ and $H$ both contain elements of order $p$ and $q$, but not of order $pq$, the entry $pq$ will only appear for $G$. E.g. $A_7$ and $S_7$ both contain involutions and elements of order $7$, but not of order $2\cdot 7$. So $2\cdot 7$ will appear for $S_7$, but not for $A_7$. While $2\cdot 5$ appears for $A_7$, but not for $S_7$, since $S_7$ contains elements of order 10.
The column sometimes also contains primes. A prime $p$ appears, if the partial augmentations of units of order $p$ are needed to  obtain the maximal restrictions on units of order $pq$. E.g.\ to obtain that there are no units of order $3\cdot 5$ for $\PSL(2,81).4a$ one needs to know the partial augmentations of units of order 3. In two cases a prime $p$ will not appear in the column, although it appears as a factor of $p\cdot q$. This will happen if either there is only one conjugacy class of elements of order $p$ in the group and thus the partial augmentations are clear, or if the characters used to obtain the information for units of order $p \cdot q$ are constant on all conjugacy classes of elements of order $p$. In the second case the partial augmentations of units of order $p$ do not influence the  calculations for elements of order $p \cdot q$ and this is also marked in the fourth column.\\
For the possibly infinite series appearing in the table, some of the groups in these series contain elements of order $p\cdot q$, others do not. For that reason these orders appear in parentheses and the edge is dotted in the prime graph.
\item The fourth column contains how one can obtain the information for units of the order given in the third column. If the order is composite this information is sufficient to obtain the maximal available information using the HeLP method. The information can be of the following types:
\begin{itemize}
\item If the group has been handled before in the literature a reference is given. If only some specific order has been handled a reference is given for this order.
\item Several cases are covered by general results in this article and in this case an internal reference is given.
\item If explicit computations have to be performed a list of characters is given. For composite orders these characters allow us to obtain the maximal possible information available via the HeLP restrictions as explained in \cref{Extended_HeLP_restrictions}. Mostly the characters are taken from the \textsf{GAP} character table library. $\chi_{n/m\alpha}$ denotes the $n$-th ordinary in the character table. This is the $\alpha$-th character of degree $m$ in that table. If the table contains only one character of degree $m$ the index $\alpha$ is omitted. E.g. the restriction obtained for units of order $2\cdot 3$ for $\PSL(2,16)$ use two ordinary characters $\chi_{2/11a}$ and $\chi_{11/17a}$ from \texttt{CharacterTable("L2(16)");}. While for units of order $2\cdot 5$ only the character $\chi_{12/17b}$ is used, which is the second character of degree 17 in the table. 
See \cref{ExplicitExample} how to use this information in the implementation of the HeLP method.
\\
Characters of the form $\varphi^r_{n/m\alpha}$ refer to Brauer characters modulo $r$ available in the \textsf{GAP} character table library. The lower indices are to be read in the same way as for ordinary characters. E.g. the character $\varphi^3_{3/55a}$ used for $\PSL(3,7).2$ refers to the third character of \texttt{CharacterTable("L3(7).2") mod 3;}.\\
In some situations where the characters available in the \textsf{GAP} character table library are not sufficient to obtain the maximal possible information induced characters are used. This happens for 
two groups containing $\PSU(3,8)$. 
See \cref{code_induced_characters}, how to do this.
\end{itemize}  
\item The fifth column contains the number of possible non-trivial partial augmentations obtained for the order in the third column. For composite orders this is the minimal possible number obtainable using the HeLP method. In case they are non-zero, they are set in bold and they are exactly the critical cases remaining to answer the Prime Graph Question. For units of prime order the information obtainable with the characters in the \textsf{GAP} character table library may sometimes not be sufficient to obtain the maximal possible restrictions, since some Brauer tables are not available in the library. In this cases the given number are marked with a star. In these situations we are however able to prove the Prime Graph Question and for this reason we do not aim at obtaining the best possible information.
\end{itemize}

{\footnotesize
 \input{helptable3}
}

\begin{remark} Note that there are more almost simple groups with 4-primary socle, but these are not 4-primary themselves and thus not listed in \cref{HeLPTable}. Namely this is the automorphism group of $\Sz(8)$, having 3 as an additional prime divisor, the automorphism groups of $\PSL(2,2^f)$ and $\PSL(2,3^f)$ and a degree $f$-extension of $\PSL(2,3^f)$, where the group either appears in the series $\PSL(2,3^f)$ described in \cref{4primaryProposition} or $3^f=3^5$. Note that in the case that $\PSL(2,3^f)$ is $4$-primary, $f$ has to be a prime if $f \geq 5$. This follows easily from \cref{4primaryProposition} and \cite[3.3.4]{Wilson}.
\end{remark}

\section*{Acknowledgements}
We are grateful to Christof Söger and Sebastian Gutsche for helping us to solve linear inequalities on computers. We thank Thomas Breuer for help with the \textsf{GAP} Character Table Library. We also want to thank the referee for improving the presentation in the article.

We moreover thank Sara Cebellan Debon for her careful reading and finding mathematical gaps in an earlier version of the paper.

\bibliographystyle{amsalpha}
\bibliography{HeLPTeil}

\end{document}

%% file: helptable3.tex
\begin{center}

\begin{longtable}{p{3.2cm} p{2cm} l p{4cm} p{1.4cm}} \caption{Results of the HeLP-method applied to the almost simple $4$-primary groups}\label{HeLPTable} \\ \toprule[1.5pt]
Group $G$  & $\Gamma(G)$ & $o(u)$ & \multicolumn{2}{c}{Characters \hfill \# n.-t. Solutions}  \\ \midrule 
\endfirsthead \toprule[1.5pt]Group $G$  &$\Gamma(G)$ & $o(u)$ & \multicolumn{2}{c}{Characters \hfill \# n.-t. Solutions}  \\ %\midrule 
\endhead \hline \multicolumn{5}{c}{continued}\\ \midrule[1.5pt]\endfoot\bottomrule[1.5pt]\endlastfoot

  \begin{minipage}[t]{0.2\textwidth} $S_7$, $S_8$, $S_9$, $S_{10}$ \end{minipage} & & & \cite[Theorem 5.1]{Symmetric} &  \\ \hline
 \begin{minipage}[t]{0.2\textwidth} $A_7$, $A_8$, $A_9$, $A_{10}$ \end{minipage} & & & \cite{SalimA7A8, SalimA9A10} & \\ \hline\hline

  \begin{minipage}[t]{0.2\textwidth} $\PSL(2,16).4$ \\ {\texttt{"L2(16).4"}}\end{minipage} &  $  \xygraph{
!{<0cm,0cm>;<-1cm,0cm>:<0cm,-1cm>::}
!{(1,0)}*+{\scriptstyle 2}="2"
!{(0,0)}*+{\scriptstyle 3}="3"
!{(1,1)}*+{\scriptstyle 5}="5"
!{(0,1)}*+{\scriptstyle 17}="17"
"2"-"3" "2"-"5" "3"-"5"} $ 
& \begin{minipage}[t]{0.1\textwidth} $2$  \\ $2\cdot 17$ \\ $3\cdot 17$ \\ $5\cdot 17$\end{minipage} & \begin{minipage}[t]{0.3\textwidth} \cref{Lemmaxy} ($\chi_{2,1b}$) \\ $\chi_{7/16a^{C17}}$ \\ $\chi_{7/16a^{C17}}$ \\ $\chi_{7/16a^{C17}}$  \end{minipage} &  \\ \hline

\begin{minipage}[t]{0.25\textwidth} $\PSL(2,16).2$ \ {\texttt{"L2(16).2"}}\end{minipage} & $\Gamma(\PSL(2,16).4)$\\ \hline

\begin{minipage}[t]{0.2\textwidth} $\PSL(2,16)$ \\ {\texttt{"L2(16)"}}\end{minipage} &  $  \xygraph{
!{<0cm,0cm>;<-1cm,0cm>:<0cm,-1cm>::}
!{(1,0)}*+{\scriptstyle 2}="2"
!{(0,0)}*+{\scriptstyle 3}="3"
!{(1,1)}*+{\scriptstyle 5}="5"
!{(0,1)}*+{\scriptstyle 17}="17"
"3"-"5"} $ 
& \begin{minipage}[t]{0.1\textwidth} $5$ \\ $2\cdot 3$ \\ $2\cdot 5$ \end{minipage} & \begin{minipage}[t]{0.3\textwidth} \cref{HertweckPSL} \\ $\chi_{2/11a}, \chi_{11/17a}$ \\ $\chi_{12/17b}$ \end{minipage} & \begin{minipage}[t]{0.1\textwidth} \phantom \\ \textbf{2} \end{minipage} \\
\hline

\end{longtable}
\end{center}